\documentclass[12pt]{amsart}
\usepackage[left=1in,top=1in,right=1in,bottom=1in]{geometry}

\usepackage[english]{babel}
\usepackage{amsmath,amsthm,amsfonts,amssymb,epsfig,xcolor,hyperref}
\usepackage{bbm,booktabs,float,mathtools,siunitx,tikz}
\usepackage[shortlabels]{enumitem}

\newtheorem{thm}{Theorem}
\newtheorem{lem}{Lemma}
\newtheorem{prop}{Proposition}
\newtheorem{cor}{Corollary}

\newcommand{\E}{\mathbf{E}}

\newcommand{\PP}{\mathbf{P}}
\newcommand{\wS}{\widetilde{S}}
\newcommand{\R}{\mathbb{R}}

\newcommand{\N}{\mathbb{N}}
\newcommand{\e}{\varepsilon}
\newcommand{\ii}{{\rm i}}
\newcommand{\oo}{{\rm o}}
\newcommand{\rd}{{\rm d}}
\newcommand{\m}{\text{meas}}

\DeclareMathOperator{\OO}{O}

\newcommand{\re}{{\rm Re}}

\newcommand\numberthis{\addtocounter{equation}{1}\tag{\theequation}}

\newcommand{\1}{\mathbf 1}

\definecolor{plum}{RGB}{141,75,196}

\begin{document}
\title{Lower Bounds for the Large Deviations of Selberg's Central Limit Theorem}

\author{Louis-Pierre Arguin}

\address{L.P. Arguin, Mathematical Institute, University of Oxford, UK and Department of Mathematics, Baruch College and Graduate Center, City University of New York, NY}
\email{louis-pierre.arguin@maths.ox.ac.uk}

\author{Emma Bailey}
\address{E. Bailey, Department of Mathematics, Graduate Center, City University of New York, NY}
\email{ebailey@gc.cuny.edu}

\begin{abstract}
  Let $\delta>0$ and $\sigma=\frac{1}{2}+\tfrac{\delta}{\log T}$. We prove that, for any $\alpha>0$ and $V\sim \alpha\log \log T$ as $T\to\infty$, 
  \[
\frac{1}{T}\text{meas}\big\{t\in [T,2T]: \log|\zeta(\sigma+\ii \tau)|>V\big\}\geq C_\alpha(\delta)\int_V^\infty \frac{e^{-y^2/\log\log T}}{\sqrt{\pi\log\log T}}\rd y,
  \]
  where $\delta$ is large enough depending on $\alpha$. The result is unconditional on the Riemann hypothesis.
  As a consequence, we recover the sharp lower bound for the moments on the critical line proved by Heap \& Soundararajan and Radziwi\l\l{ }\& Soundararajan. 
  The constant $C_\alpha(\delta)$ is explicit and is compared to the one conjectured by Keating \& Snaith for the moments. 
\end{abstract}

\date{October 29, 2024}

\maketitle


\section{Main Results}
The Riemann zeta function is defined on the half-plane $\re \ s>1$ by 
\[
\zeta(s)=\sum_{n\geq 1}n^{-s}=\prod_{p\ \text{prime}}(1-p^{-s})^{-1}.
\]
It can be continued to a meromorphic function on the whole complex plane $\mathbb C$ with a pole at $s=1$. 
The Riemann hypothesis states that all non-trivial zeros of the function lie on the critical line $\re \ s=1/2$.
This paper is concerned with the distribution of large values of the function on the critical line and slightly to its right. 

A good point of reference for the distribution of values of the function on the line is Selberg's Central Limit Theorem:
if $\tau$ is a random point sampled uniformly in $[T,2T]$, then we have the following convergence in distribution
\begin{equation}
  \label{eqn: selberg CLT}
  \frac{\log |\zeta(1/2+\ii\tau)|}{\frac{1}{2}\sqrt{\log\log T}}\to \mathcal N(0,1), \quad T\to\infty,
\end{equation}
for $\mathcal N(0,1)$, a standard Gaussian random variable. 
It is not hard to see from the proof, see for example \cite{radsou17}, that the result also holds slightly off-axis, i.e. $\re \ s=\frac{1}{2}+\frac{\delta}{\log T}$, $\delta>0$.
An important question, because of its relation to the moments of $\zeta$ and to the Lindel\"of hypothesis, is to see if the distribution of the values of $\log|\zeta|$ remains
Gaussian-like for values of the order of $\log\log T$, thus well beyond the standard deviation $\sqrt{\log\log T}$. 
It was established in \cite{ArgBai23} for $V\sim \alpha\log\log T$, $\alpha\in (0,2)$, that
\begin{equation}
  \label{eqn: AB UB}
  \PP(\log |\zeta(1/2+\ii\tau)|>V)\leq K_\alpha \int_V^{\infty} \frac{e^{-y^2/\log\log T}}{\sqrt{\pi\log\log T}}\rd y,
\end{equation}
for some constant $K_\alpha$ that depends on $\alpha$ but not $T$, and where $\PP$ stands for the uniform probability on $\tau$.
This in turn implied another proof for a sharp upper bound for the fractional moments of $\zeta$ between $0$ and $4$ as first proved in \cite{hearadsou19}.
The large deviations of the imaginary part of $\log \zeta$ for a wide range values are investigated in \cite{Dob24}.

Here, we prove an unconditional lower bound off-axis matching \eqref{eqn: AB UB} up to constant:
\begin{thm}
  \label{thm: selberg LB}
  Let $\alpha>0$ and $V\sim \alpha \log\log T$. Consider $\tau$ sampled uniformly on $[T,2T]$.
  There exists $\delta^\star(\alpha)$ such that if $\delta>\delta^\star(\alpha)$, we have
  \[
  \PP\Big(\log |\zeta(1/2+\delta/\log T+\ii \tau)|>V\Big) \geq C_\alpha(\delta)\int_{V}^{\infty} \frac{e^{-y^2/\log\log T}}{\sqrt{\pi \log\log T}}\rd y, 
  \]
  for some $C_\alpha(\delta)>0$.
\end{thm}
The constant $C_\alpha(\delta)$ is explicit, see Equation \eqref{eqn: heuristic2}.
So is $\delta^\star$, as given in Equation \eqref{eqn: delta star}. 
For $\alpha$ tending to $\infty$, $\delta^\star(\alpha)$ is of order $\log \alpha$, whereas $\delta^\star(\alpha)$ remains bounded away from $0$ as $\alpha\to 0$.
(For technical reasons, $\delta$ is redefined in \eqref{eqn: sigma}, but this does not change the asymptotic behavior.)
Our method of proof does not seem to apply at $\delta=0$, although it does imply sharp lower bound on the moments on the critical line.
More precisely, it was conjectured by Keating \& Snaith \cite{keasna00a} that for all $k\geq 0$
\begin{equation}
  \label{eqn: moments}
  \frac{1}{T}\int_T^{2T}|\zeta(1/2+\ii t)|^{2k}\rd t=\E\left[|\zeta(1/2+\ii \tau)|^{2k}\right]\sim C_k\  (\log T)^{k^2},
\end{equation}
where $C_k=a_kf_k$. The {\it arithmetic part} $a_k$  is of the form
\begin{equation}
  \label{eqn: a_k}
  a_k=\prod_{p}(1-p^{-1})^{k^2}\sum_{m\geq 0} \left(\frac{\Gamma(k+m)}{m!\Gamma(k)}\right)^2p^{-m},
\end{equation}
(see also Equation \eqref{eqn: divisor}), and the {\it random matrix part} is given by
\begin{equation}
  \label{eqn: f_k}
  f_k=\frac{G^2(1+k)}{G(1+2k)},
\end{equation}
where $G$ is the Barnes $G$-function.
Theorem \ref{thm: selberg LB} yields:
\begin{cor}
  \label{cor: moments}
  For $k\geq 0$ fixed, we have
  \[
  \E\left[|\zeta(1/2+\ii \tau)|^{2k}\right]\geq c_k (\log T)^{k^2},
  \]
  where the constant $c_k$ is proportional to  
  \begin{equation}
    \label{eqn: c_k}
    a_k  e^{\gamma k^2}(k+1)^{-38k^2}.
  \end{equation}
\end{cor}
Sharp lower bounds of this kind were proved in \cite{heapsou22} and \cite{radsou13}.
The constant obtained by Heap \& Soundararajan is better for large $k$ as it behaves like $k^{-2k^2}$.
For $k$ approaching $0$, their constant tends to $0$ whereas the one in Corollary \ref{cor: moments} is bounded away from it, as it should. Although, as mentioned in their paper, a modification of their argument would lead to the correct behavior for small $k$.
Note that the constant $c_k$ includes explicitly the arithmetic part $a_k$, and this is in fact necessary in our argument.
However, this does not change the asymptotic behavior for large $k$, since as shown in Conrey \& Gonek \cite{congon01}, $\log a_k\sim -k^2\log\log k$ as $k\to\infty$.
The correct asymptotic behavior of $f_k$ is $\log f_k\sim -k^2 \log k$.
It is unclear to us at this time how to improve the constant $c_k$, though a better approximation of indicator functions in terms of polynomials seems necessary, cf.~Lemma \ref{lem: approx} and the discussion following Equation \eqref{eqn: heuristic2}. 
It is important to note that the same constant $C_k$ as for the moments should appear in the large deviation of $\log |\zeta|$. Indeed, it was conjectured by Radziwi\l\l{ } that
for any $\alpha>0$ and $V\sim \alpha \log\log T$ as $T\to\infty$,
\begin{equation}
  \label{eqn: radziwill}
  \PP(\log |\zeta(1/2+\ii\tau)|>V)\sim C_\alpha \int_V^\infty \frac{e^{-y^2/\log\log T}}{\sqrt{\pi\log\log T}} \rd y.
\end{equation}

The paper is organized as follows. The proof of Theorem \ref{thm: selberg LB} is explained in Section \ref{sect: structure}. It is divided in four propositions that
are proved in Section \ref{sect: prop}. Corollary \ref{cor: moments} is proved in Section \ref{sect: cor}.
Finally, the appendix contains results needed in the proof that have appeared elsewhere in the literature. One notable exception is the mollification result stated in Lemma \ref{lem:mollification} that may be of independent interest. It was developed with one of the authors and Bourgade \& Radziwi\l\l{ }in an earlier version of \cite{argbourad23}.

\medskip

{\bf Notation.}
We write $f(T)=\OO(g(T))$ whenever $\limsup_{T\to\infty} |f(T)/g(T)| <\infty$ with other parameters fixed, such as $\alpha$ and $k$.
We also use Vinogradov's notation $f(T)\ll g(T)$ whenever $f(T)=\OO(g(T))$, and $f(T)\asymp g(T)$ if $f(T)\ll g(T)$ and $f(T)\gg g(T)$.

\medskip

{\bf Acknowledgement.}
The authors are grateful to Winston Heap for his comments on the first version of this paper. 
L.-P.~A.~thanks Paul Bourgade and Maksym Radziwi\l\l{ }for numerous insightful discussions on the subject, especially on the proof of Lemma \ref{lem:mollification}.  
L.-P.~A.~is supported in part by the NSF award DMS 2153803. 

\section{Structure of the Proof of Theorem \ref{thm: selberg LB}}
\label{sect: structure}
\subsection{Some definitions}
As in \cite{ArgBai23}, the proof relies on approximating $\log|\zeta|$ by Dirichlet polynomials of the form:
\begin{equation}
  \label{eqn: S}
  S_\ell=\sum_{T_0<p \leq T_\ell} \frac{\re\ p^{-\ii \tau}}{p^{\sigma}}+\frac{\re \ p^{-2\ii\tau}}{2p^{2\sigma}}, \quad \ell>0.
\end{equation}
Throughout we use the probability convention of dropping $\tau$ in the notation. Also, $\sigma$ will be fixed and its dependence often omitted.
The small primes need to be treated with care to get a good estimate on the constant. 
For the primes below $T_0$, we consider $|\mathcal M_0^{-1}|$ where
\begin{equation}
  \label{eqn: M0}
  \mathcal M_0=\mathcal{M}_{0}(\sigma+\ii \tau)= \prod_{p\leq T_0}\Big(1-p^{-\sigma-\ii \tau}\Big)=\sum_{\substack{p | m \Rightarrow p \leq T_0}} \frac{\mu(m)}{m^{\sigma+\ii  \tau}}.
\end{equation}
The parameter $T_0$ cannot be too large to ensure that $\mathcal{M}_{0}$ stays relatively short. An admissible choice turns out to be:
\begin{equation}
  \label{eqn: T0}
  T_0=\log\log T.
\end{equation}
The parameters $T_\ell$ are chosen to approach $T$ fairly quickly:
\begin{equation}
  \label{eqn: Tell}
  T_\ell=T^{\frac{1}{(\log_{\ell+1}T)^{\mathfrak s}}}, \ 1\leq \ell\leq \mathcal L.
\end{equation}
The constant $\mathcal L$ is the largest $\ell$ for which $\log_{\ell+1} T>e$. By definition, we thus have $e<\log_{\mathcal L+1} T\leq e^2$ and $1<\log_{\mathcal L+2} T\leq e$.
The parameter $\mathfrak s$ plays an important role in the analysis to estimate the constant $C_\alpha(\delta)$. The technical analysis requires $\mathfrak s=\mathfrak s(\alpha)$ to be not too small, to be precise:
\begin{equation}
  \label{eqn: s}
  \frac{(\log_{\mathcal L+1}T)^{\mathfrak s}}{(\mathfrak s \log_{\mathcal L+2}T)^{35}}>K(\alpha+1)^{35},
\end{equation}
where $K$ is a large absolute constant independent of $\alpha$. The power of $35$ is imposed during the proof of Proposition~\ref{prop: S to G} (in particular in Lemma~\ref{lem: approx}),  see \eqref{eqn: choice of s} and the surrounding discussion for further details.

It is convenient for the analysis to parametrize $\sigma>1/2$ in the following form:
\begin{equation}
  \label{eqn: sigma}
  \sigma=\frac{1}{2}+\frac{\delta}{\log T_\mathcal L}.
\end{equation}

\subsection{Heuristic}
By informally expanding the Euler product in the definition of $\zeta$, we expect that $\log|\zeta|$ at $\sigma+\ii \tau$ should behave like 
\[
\log|\zeta(\sigma+\ii \tau)|\approx S_{\mathcal L} + \log |\mathcal M_0|^{-1}.
\]
Moreover, since $S_\ell$ is a sum of weakly correlated terms, it should behave like a Gaussian random variable for large $T$. 
With this in mind, we define the random variables
\begin{equation}
  \label{eqn: Z_l}
  \mathcal Z_\ell=\sum_{T_0 <p \leq T_\ell} \frac{Z_p}{p^{\sigma}}+ \frac{Z_p^2}{2p^{2\sigma}}, \quad 1\leq \ell \leq \mathcal L.
\end{equation}
We effectively replaced $p^{-\ii \tau}$ by $Z_p$ where $(Z_p, \text{p prime})$ are centered Gaussian random variables with variance $1/2$.
The contribution of very small primes is not Gaussian even in the limit $T\to\infty$. Instead, for the small primes, we substitute the phase $p^{-\ii \tau}$ in by $e^{\ii \theta_p}$, where $(\theta_p, \text{p prime})$ are IID uniform on $[0,2\pi]$:
\begin{equation}
  \label{eqn: Z0}
  \mathcal Z_0=\log \prod_{p\leq T_0} \Big|1-\frac{e^{\ii \theta_p}}{p^\sigma}\Big|^{-1}=-\sum_{p\leq T_0} \sum_{k\geq 1}\frac{\cos k\theta_p}{kp^{k\sigma}}.
\end{equation}
We thus expect for $V\sim \alpha \log\log T$, $\alpha>0$, that
\begin{equation}
  \label{eqn: heuristic approx}
  \PP(\log |\zeta(\sigma+\ii \tau)|>V)\approx \PP(\mathcal Z_0+\mathcal Z_\mathcal L>V).
\end{equation}
The variance of $\mathcal Z_\mathcal L$ is not exactly $1/2(\log \log T_\ell -\log\log T_0)$ as one would expect from Mertens's theorem, since $\sigma>1/2$.
In fact, we have for any $1\leq \ell\leq \mathcal L$ by the Prime Number Theorem
\begin{align*}
  v_\ell^2=\E[\mathcal Z_\ell^2]=\sum_{T_0< p\leq T_\ell}\Big(\frac{1}{2p^{2\sigma}}+\frac{1}{8p^{4\sigma}}\Big)&=\frac{1}{2}\int_{\log T_0}^{\log T_\ell} \frac{e^{-(2\sigma-1)u}}{u}\rd u+\OO(e^{-c\sqrt{\log T_0}})\\
  &=\frac{1}{2}\left(\log \log T_\ell -\log\log T_0\right)-\delta \frac{\log T_\ell}{\log T_\mathcal L}+\OO(e^{-c\sqrt{\log T_0}}).\numberthis \label{eqn: variance}
\end{align*}
The conclusion of Theorem \ref{thm: selberg LB} is then simply the result of a Gaussian estimate, conditioning on $\mathcal Z_0$, and of expanding $v_\mathcal L$:
\begin{align*}
  \PP(\mathcal Z_\mathcal L+\mathcal Z_0>V)
  &\asymp \E\Big[\frac{e^{-(V-\mathcal Z_0)^2/(2v_{\mathcal L}^2)}}{\alpha\sqrt{\log\log T}}\Big]\\
  &\asymp \left(\frac{\E[e^{2\alpha \mathcal Z_0}]}{(\log T_0)^{\alpha^2}} \cdot e^{-2\alpha^2\delta}\cdot (\log_{\mathcal L+1} T)^{-\alpha^2\mathfrak s}\right) \frac{e^{-V^2/\log\log T}}{\alpha\sqrt{\log\log T}} \ .\numberthis\label{eqn: heuristic}
\end{align*}
The term $(\log T_0)^{-\alpha^2}\E[e^{2\alpha \mathcal Z_0}]$ turns out to be exactly the arithmetic factor $e^{\gamma \alpha^2}a_\alpha$ by a short computation on the random model, cf.~Lemma \ref{lem: ak}.
We therefore obtain the lower bound claimed in Theorem \ref{thm: selberg LB} with 
\begin{equation}
  \label{eqn: heuristic2}
  C_\alpha(\delta)=e^{\gamma \alpha^2}a_{\alpha}\cdot e^{-2\alpha^2\delta}\cdot (\alpha+1)^{-36\alpha^2},
\end{equation}
by taking $( \log_{\mathcal L+1} T)^\mathfrak s=K(\alpha+1)^{36}$, which satisfies the bound \eqref{eqn: s}.  The significance of each factor constituting $C_\alpha(\delta)$ is clear from the heuristic. The arithmetic factor $a_\alpha$ is produced by the small primes. The term $e^{-2\alpha^2\delta}$ is due to the shift off-axis and should be $0$ on the critical line (perhaps by assuming the Riemann hypothesis). Finally, the factor $(\alpha+1)^{-36\alpha^2}$ is due to the cutoff of the large prime at $T^{1/(\log_{\mathcal L+1}T)^\mathfrak s}$. The optimal cutoff should be expected to be $T^{c/\alpha}$ for some small constant $c>0$. Such a cutoff would yield a factor of $(\alpha+1)^{-\alpha^2}$ corresponding to the asymptotics of $f_\alpha$ in \eqref{eqn: f_k} predicted by random matrix theory.

\subsection{Proof of Theorem \ref{thm: selberg LB}}
We now outline the roadmap to establish a rigorous lower bound for \eqref{eqn: heuristic approx}.
An approximation of $\log |\zeta|$ in terms of the partial sums $S_\ell$ and of the Gaussian sums $\mathcal Z_\ell$ can be made precise on a {\it good event} where the sums are constrained.
Specifically, we define
\begin{equation}
  \label{eqn: G0}
  G_0=\{|\mathcal M_0|^2\in [L_0,U_0]\}.
\end{equation}
and
\begin{equation}
  \label{eqn: G} 
  G_\ell=G_{\ell-1}\cap \{L_\ell(S_0)\leq S_\ell \leq U_\ell(S_0)\}, 1\leq \ell \leq \mathcal L,
\end{equation}
where the upper and lower \emph{barriers} $L_\ell(z), U_\ell(z)$ are defined and discussed below in \eqref{eqn: barriers}, and $S_0 = -\log |\mathcal{M}_0|$, following the definition~\eqref{eqn: M0}. The initial good event $G_0$ is defined in terms of $\mathcal M_0$, since it is a short Dirichlet polynomial, cf.~\eqref{eqn: Y0}. We take 
\begin{equation}
  \label{eqn: first barrier}
  L_0=\exp(-2\alpha \log \log T_0)\qquad U_0=\exp(-2\alpha \log \log T_0 +2\sqrt{\log\log T_0}).
\end{equation}
The factor of $2$ in \eqref{eqn: first barrier} is a direct consequence of taking the squared polynomial in \eqref{eqn: G0}.
The final event of interest is 
\begin{equation}
  \label{eqn: G(V)}
  G(V)= \{\log|\mathcal M_0|^{-1}+S_\mathcal L>V\}\cap G_\mathcal L.
\end{equation}

We define the corresponding events for the sums $\mathcal Z_\ell$,
\begin{align*}
  \mathcal G_0&=\{\mathcal Z_0-\alpha\log\log T_0\in [0,\sqrt{\log\log T_0}]\}\\ 
  \mathcal G_\ell&=\mathcal G_{\ell-1}\cup \{L_\ell(\mathcal{Z}_0)\leq \mathcal Z_\ell\leq  U_\ell(\mathcal{Z}_0)\}, \ 1\leq \ell \leq \mathcal L, 
\end{align*}
and 
\begin{equation}
  \label{eqn: G(V) random}
  \mathcal G(V)= \{\mathcal Z_0+\mathcal Z_\mathcal L>V\}\cap \mathcal G_\mathcal L.
\end{equation}
In $G_0$ and $\mathcal{G}_0$, the upper bound is natural, coming from the standard deviation in both cases.  It is convenient to assume that the recentered sum in case is positive, and we use this fact in the proof of Proposition~\ref{prop: G}, cf.\;\eqref{eq: upperbarrier_bound}. The (random) upper and lower barriers $U_\ell(z)$ and $L_\ell(z)$ are defined by
\begin{equation}
  \label{eqn: barriers}
  \begin{aligned}
    L_\ell(z)&=m(z)(\log\log T_\ell-\log\log T_0)-\mathcal{C}\log_{\ell+2}T\\
    U_\ell(z)&=m(z)(\log\log T_\ell-\log\log T_0)+\mathcal{B}\log_{\ell+2}T 
  \end{aligned}
\end{equation}where the \emph{slope} $m(z)$ is defined as
\begin{equation}
  \label{eqn: slope}
  m(z) = \frac{\alpha\log\log T-z}{\log\log T_\mathcal{L} - \log\log T_0}.
\end{equation}
where $\log_{\ell}$ stands for the logarithm iterated $\ell$ times.  In each case, the slope $m(z)$ will be evaluated at the (random) initial position, i.e., $S_0$ or $\mathcal{Z}_0$. Taking the barriers randomly is necessitated by separating out the contribution from the small primes.  On first reading, one can think of the slope $m(z)\approx \alpha$ (since typically $z$ is like $\alpha \log\log T_0$ in each case).  However, we need to take in to account the variance coming from the initial random sums which informs the slope depending on the random position of $\mathcal{Z}_0$ or $S_0$, see in particular the remark following \eqref{eq: lowerbarrier_bound}. See also Figure~\ref{fig:gaussianModel}, where two different initial values $\mathcal{Z}_0$ and $\mathcal{Z}'_0$ both reach level $V$ but elicits two different slopes. 

The constants $\mathcal B$ will need to be larger than $\alpha\mathfrak s$ and $\mathcal C$ a sufficiently large constant
\begin{equation}
  \label{eqn: BC}
  \mathcal B=\alpha\mathfrak s+\frac{1}{\alpha}\qquad \mathcal C=100.
\end{equation}
We include the term $1/\alpha$ in $\mathcal{B}$ to allow for small $\alpha$, cf. Section~\ref{sect: propG}.

\begin{figure}[htbp]
  \centering
  \begin{tikzpicture}
    \node[anchor=south west, inner sep=0] (image) at (0,0) {\includegraphics[width=0.8\textwidth]{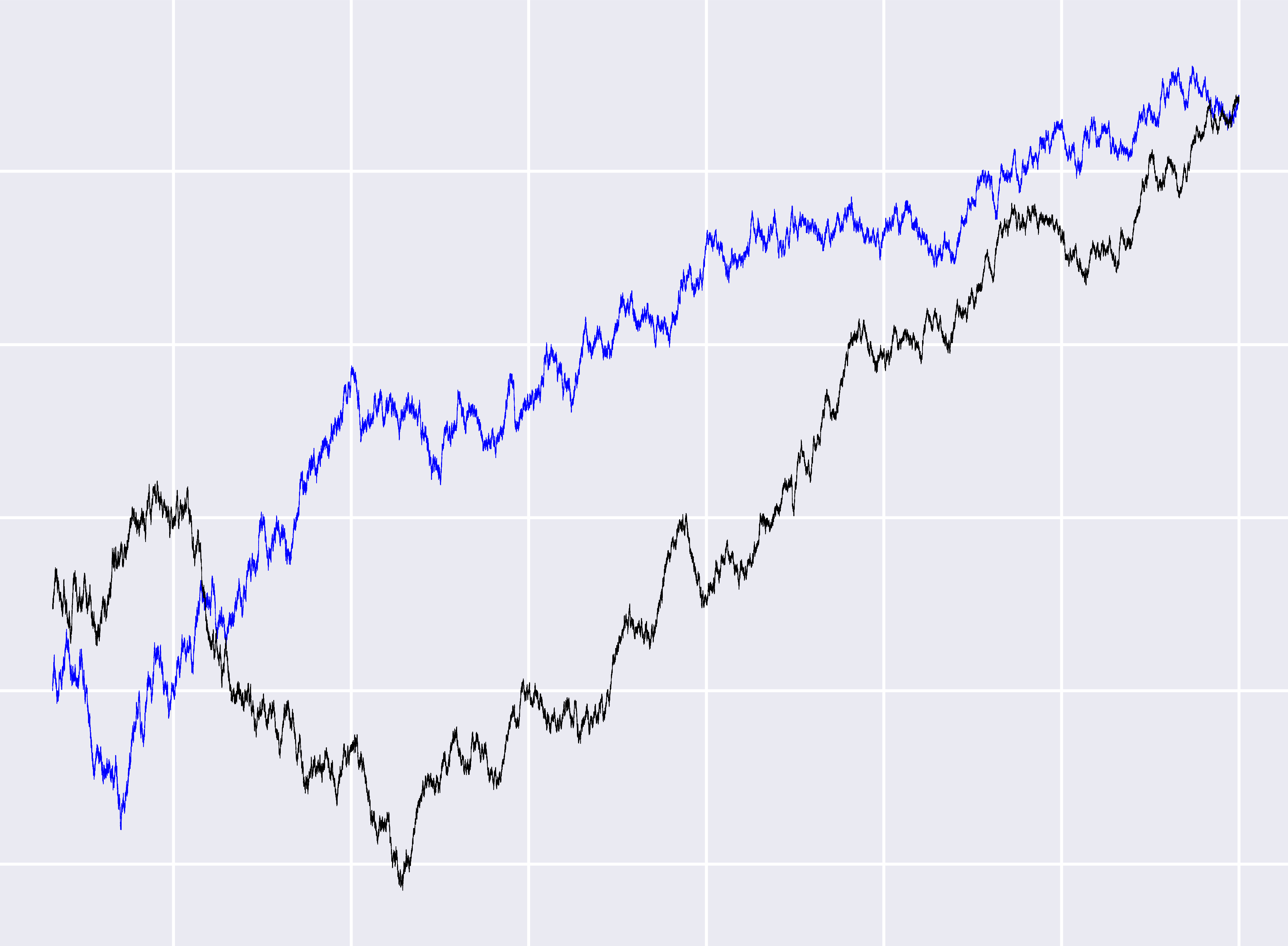}};
    \draw[black, thick,->] (image.south west) -- (image.south east);
    \draw[black, thick,->] (image.south west) -- (image.north west);
    \draw[black, very thick] (0,2.65) -- (-0.2,2.65);
    \node at (-.8, 2.65) {$\mathcal{Z}_0$};
    \draw[black, very thick] (0,3.52) -- (-0.2,3.52);
    \node at (-.8, 3.52) {$\mathcal{Z}_0'$};
    \draw[black, very thick] (0,8.7) -- (-0.2,8.7);
    \node at (-.8, 8.7) {$V$};
    \draw[black, very thick] (0.55,0) -- (0.55,-0.2);
    \node at (0.55,-0.5) {$\log\log T_0$};
    \draw[black, very thick] (12.7,0) -- (12.7,-0.2);
    \node at (12.9, -0.5) {$\log \log T_\mathcal{L}$};
    \draw[black, very thick] (0.55,3.5) -- (12.7,8.7);
    \draw[blue, very thick] (0.55,2.65) -- (12.7,8.7);
    \draw[fill=white, draw=black, line width=0.5pt] (5,8.5) rectangle (10,9.5);
    \node at (7.5,9) {$f_z(T_j)=z+m(z)\log\frac{\log T_j}{\log T_0}$};
  \end{tikzpicture}
  \caption{Illustration of the random sums starting at two distinct initial values, both ending at the same point $V$.   The lines $f_z(T_j)$ have (distinct) slopes $m(z)$ for $z=\mathcal{Z}_0, \mathcal{Z}'_0$.}
  \label{fig:gaussianModel}
\end{figure}
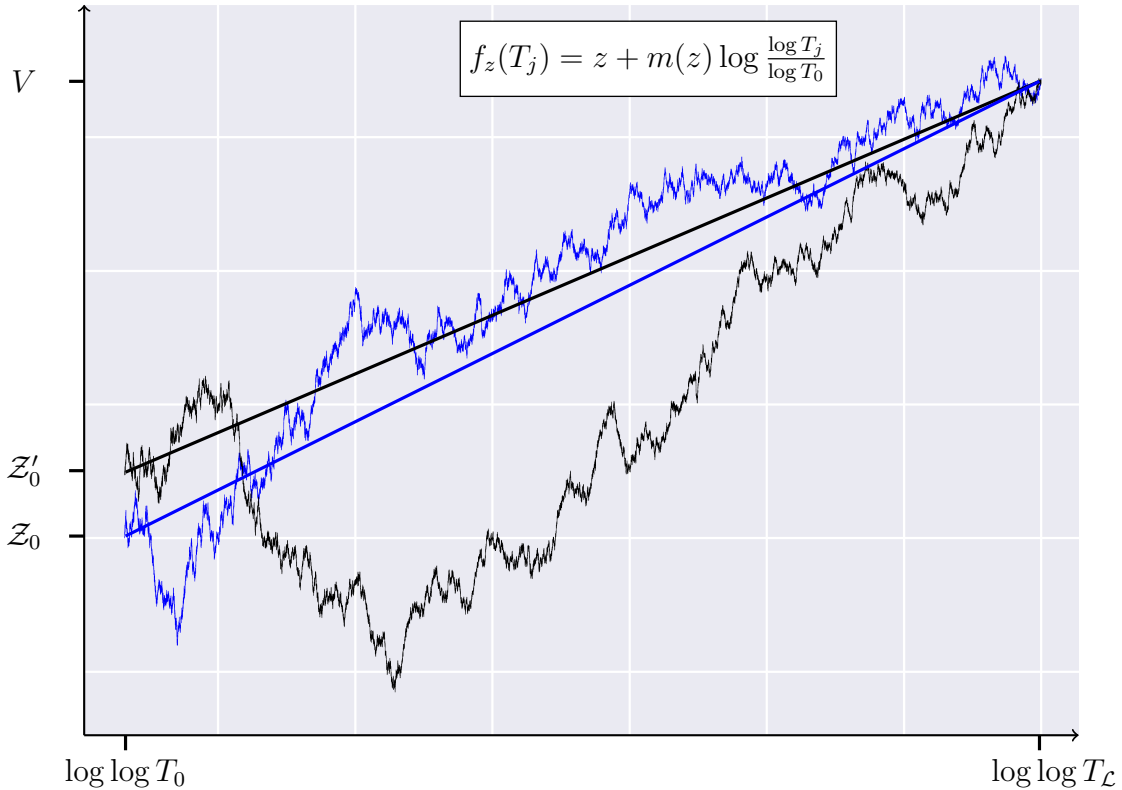

The rationale behind the choice of restrictions is that either sum $S_\ell$ or $\mathcal{Z}_\ell$, whose total value is $V\sim \alpha\log\log T$, should stay close to the linear interpolation with slope $m(z)$ with some fluctuations that are getting smaller with $\ell$.  It is not hard to first establish this fact for the random model:
\begin{prop}
  \label{prop: G}
  For fixed $\alpha>0$ and $V\sim \alpha \log\log T$, the probability of the event $\mathcal G(V)$ defined in \eqref{eqn: G(V)} is 
  \[
  \PP(\mathcal G(V))\asymp \PP(\mathcal Z_0+\mathcal Z_\mathcal L>V).
  \]
\end{prop}
To prove this for the partial sums $S_\ell$ is much trickier. To do so, we follow the idea of \cite{argbourad20} and \cite{argbourad23} and approximate the events involving $S_\ell$ and $\mathcal M_0$ in terms of short Dirichlet polynomials.
This allows us to explicitly compare the probability of the event $G(V)$ with the event $\mathcal G(V)$ of the random model.
We ultimately show
\begin{prop}
  \label{prop: S to G}
  For fixed $\alpha>0$ and $V\sim\alpha \log\log T$, the probability of the event $\mathcal G(V)$ defined in \eqref{eqn: G(V) random} is 
  \begin{equation}
    \label{eqn: S to G lower}
    \PP(G(V))\asymp \PP(\mathcal Z_0+\mathcal Z_\mathcal L>V).
  \end{equation}
\end{prop}
It remains to connect the event $G(V)$ to $\log |\zeta|$. This is done by defining the following short mollifiers:
\begin{equation}
  \label{eqn: M}
  \mathcal{M}_\ell = \sum_{\substack{p | m \Rightarrow p \in (T_{\ell-1},T_\ell] \\ \Omega_\ell(m) \leq \mu_{\ell}}} \frac{\mu(m)}{m^{\sigma+\ii\tau}},\quad 1\leq \ell\leq \mathcal L,
\end{equation}
where $\Omega_\ell(m)$ is the number of prime factors of $m$ in the interval $ (T_{\ell-1},T_\ell]$, and we choose 
\begin{equation}
  \label{eqn: mu}
  \qquad  \mu_\ell=100(\alpha\vee 1)(\log\log T_{\ell} - \log\log T_{\ell - 1}),\quad 1\leq \ell\leq \mathcal L.
\end{equation}
This is to ensure the length of the mollifiers is not too long. Indeed, taking the product over $\ell$, we set (recalling the definition of $\mathcal{M}_0$ from \eqref{eqn: M0}) 
\begin{equation}
  \label{eqn: M total}
  \mathcal M=\prod_{0\leq\ell\leq \mathcal L}\mathcal M_\ell.
\end{equation}
Then, the definition of $\mu_\ell$ and the choice of $\mathfrak s$ in \eqref{eqn: s} implies that $\mathcal M$ is supported on integers smaller than
\begin{equation}
  \label{eqn: shortM}
  \prod_{0\leq \ell\leq \mathcal L} (T_\ell)^{\mu_\ell}\ll  T^{200(\alpha+ 1)\mathfrak s(\log_{\mathcal L+1}T)^{1-\mathfrak{s}}}\leq T^{\frac{1}{K(\alpha+1)^{34}}}<T^{1/100},
\end{equation}
for a sufficiently large $K$. 
The next proposition shows that $\zeta\mathcal M$ is close to $1$ on the event $G(V)$.
\begin{prop}
  \label{prop: zeta to M}
  Let $\e\in(0,1)$. For fixed $\alpha>0$ and $V\sim\alpha \log\log T$, we have
  \[
  \PP(\log |\zeta|>V)
  \gg \PP\Big(\{ \log |\mathcal M^{-1}|>V-\log(1-\e)\}\;\cap\; G(V)\Big)
  -\frac{1}{\e^2}\left(\frac{e^{-2\delta}}{\delta}+e^{-\mu_\mathcal L/10}\right)\PP(\mathcal Z_0+\mathcal Z_\mathcal L>V).
  \]
\end{prop}
The proposition establishes an interplay between $\e$ and $\delta$: if $\e$ is close to $0$ or $1$, then the expression could be negative.
In the proof of Theorem \ref{thm: selberg LB}, we find an optimal $\e$ and derive the threshold $\delta^\star$ from it. 
The last proposition relates the inverse of the mollifier to the partial sums, which can then be compared to $\PP(\mathcal Z_0+\mathcal Z_\mathcal L>V)$ using Propositions \ref{prop: S to G} and \ref{prop: G}.
\begin{prop}
  \label{prop: M to S}
  For fixed $\alpha>0$ and $V\sim\alpha \log\log T$, we have for the event $G(V)$ in \eqref{eqn: G(V)}
  \[
  \PP\left(\{\log |\mathcal M^{-1}|>V\} \cap G(V)\right)
  \gg  \PP\left(\mathcal Z_0+\mathcal Z_\mathcal L>V\right) .
  \]
\end{prop}
The above propositions combine to give a proof of the theorem.
\begin{proof}[Proof of Theorem \ref{thm: selberg LB}]
  Consider $V$ with $V\sim\alpha \log\log T$. 
  For some $\e\in (0,1)$ to be picked below, define $V'=V-\log(1-\e)$.
  Then, Propositions \ref{prop: zeta to M} and \ref{prop: M to S} directly imply
  \[
  \PP(\log |\zeta|>V)\gg \PP(\mathcal Z_0+\mathcal Z_\mathcal L >V-\log(1-\e))-\frac{1}{\e^2}\left(\frac{e^{-2\delta}}{\delta}+e^{-\mu_\mathcal L/10}\right)\PP(\mathcal Z_0+\mathcal Z_\mathcal L >V).
  \]
  We have from \eqref{eqn: heuristic2} that this is
  \[
  \gg C_\alpha(\delta) \frac{e^{-V^2/\log\log T}}{\alpha \sqrt{\log\log T}}\Big\{(1-\e)^{2\alpha}- \frac{1}{\e^2}\Big(\frac{e^{-2\delta}}{\delta}+e^{-\mu_\mathcal L/10}\Big)\Big\}.
  \]
  For a fixed $\alpha$, we are looking for the smallest $\delta$ that makes  the quantity in the parenthesis positive. The function $\e^2(1-\e)^{2\alpha}$ is maximal at $\e^\star=1/(\alpha+1)$. We thus pick $\delta^\star$ to be equal to a fraction (say $e^{-1}$) of  $\e^2(1-\e)^{2\alpha}$:
  \[
  \frac{e^{-2\delta^\star}}{\delta^\star}=\frac{1}{e}{\e^\star}^2(1-\e^\star)^{2\alpha}=\frac{1}{e}\frac{\alpha^{2\alpha}}{(\alpha+1)^{2(\alpha+1)}},
  \]
  which gives
  \begin{equation}
    \label{eqn: delta star}
    \delta^\star+\frac{1}{2}\log \delta^\star=\log(\alpha+1)+\alpha\log (\alpha^{-1}+1)+\frac{1}{2}.
  \end{equation}
  Since $\frac{e^{-2\delta}}{\delta}$ is a decreasing function of $\delta$ and that $(1-\e^\star)^{2\alpha}\asymp 1$, we get that for $\delta>\delta^\star$ the above is
  \[
  \gg C_\alpha(\delta) \frac{e^{-V^2/\log\log T}}{\alpha \sqrt{\log\log T}}\Big(\frac{\alpha}{\alpha+1}\Big)^{2\alpha}\gg C_\alpha(\delta)\frac{e^{-V^2/\log\log T}}{\alpha \sqrt{\log\log T}}.
  \]
  The term $e^{-\mu_\mathcal L}$ does not contribute by the choice of $\mu_\ell$ in \eqref{eqn: mu} by taking $\mathfrak s$ large enough independently of $\alpha$, as mentioned below \eqref{eqn: s}.
\end{proof}

\section{Proof of the Propositions}
\label{sect: prop}
Throughout the proofs, we use the following notations for conciseness:
\[
t=\log\log T,
\]
and
\[
t_0=\log\log T_0=\log_4 T\qquad t_\ell=\log\log T_\ell=t-\mathfrak s \log_\ell t, \ 1\leq \ell \leq \mathcal L.
\]

\subsection{Proof of Proposition \ref{prop: G}}\label{sect: propG}
The upper bound is immediate by dropping the event $\mathcal{G}_\mathcal{L}$, so we focus on the lower bound. Decomposing the event $\mathcal{G}_\mathcal{L}$ we inductively have
\begin{align}
  \PP(\mathcal{G}(V))&=\PP(\{\mathcal{Z}_{0}+\mathcal{Z}_\mathcal{L} > V\}\cap \mathcal{G}_0)- \sum_{\ell=0}^{\mathcal{L}-1}\PP(\{\mathcal{Z}_{0}+\mathcal{Z}_\mathcal{L} > V\} \cap \mathcal{G}_\ell \cap \mathcal{G}_{\ell+1}^c).\label{eq:prop2_expansion}
\end{align}
We subsequently bound the summands on the right-hand side of \eqref{eq:prop2_expansion} for $V\sim\alpha\log\log T=\alpha t$.  Separately considering the two possibilities for $ \mathcal{G}_{\ell+1}^c$, we have for the upper barrier
\begin{align}
  \PP(\{\mathcal{Z}_{0}+\mathcal{Z}_\mathcal{L} > V\}\cap\mathcal{G}_{\ell}\cap\{\mathcal{Z}_{\ell+1}>U_{\ell+1}(\mathcal{Z}_0)\})
  & \leq \PP\{\mathcal{Z}_{\ell+1}+\mathcal{Z}_0>U_{\ell+1}(\alpha t_0)+\alpha t_0\})\label{eq: upperbarrier_bound}\\
  &\asymp \frac{1}{\alpha\sqrt{t}}\mathbf{E}\left[\exp\left(-\frac{(U_{\ell+1}(\alpha t_0)+\alpha t_0-\mathcal{Z}_0)^2}{2v_{\ell+1}^2}\right)\right]\nonumber
\end{align}
akin to \eqref{eqn: heuristic}. The first bound uses positivity of the recentered sum. Unfolding and using the definition of $U_{\ell+1}(z)$ in \eqref{eqn: barriers}, we have that the above is bounded by
\begin{align*}
  &\ll \frac{1}{\alpha\sqrt{t}}\mathbf{E}\left[\exp\left(-\frac{1}{2v_{\ell+1}^2}\left(\frac{\alpha}{t_\mathcal{L}-t_0}(t_{\ell+1}\left(t-t_0)-t_0(t-t_\mathcal{L})\right)+\mathcal{B}\log_{\ell+1}t-\mathcal{Z}_0\right)^2\right)\right]\\
  &\ll\frac{1}{\alpha\sqrt{t}}e^{-\alpha^2t}e^{-\alpha^2t_0}\mathbb{E}\left[e^{2\alpha\mathcal{Z}_0}\right] e^{-2\alpha^2\delta e^{t_{\ell+1}-t_\mathcal{L}}-2\alpha^2\mathfrak{s}\log_\mathcal{L}t}\cdot e^{(\alpha^2\mathfrak{s}-2\alpha\mathcal{B})\log_{\ell+1}t} \\
  &\ll \PP(\mathcal{Z}_0 + \mathcal{Z}_\mathcal{L} > \alpha t) \cdot e^{(\alpha^2\mathfrak{s} - 2\alpha\mathcal{B})\log_{\ell+1}t},
\end{align*}
the final bound achieved by using Lemma~\ref{lem: ak} and comparing to \eqref{eqn: heuristic}.
The coefficient in the exponent involving $\log_{\ell+1}t$ is strictly negative by choosing $\mathcal{B}=\alpha\mathfrak{s}+1/\alpha$ (cf.~Equation \eqref{eqn: BC}), and the implicit constant can be taken independent of $\alpha$.  

To complete the bound, we estimate the exceedance of the lower barrier
\begin{equation}\label{eq: lowerbarrier_bound}
 \PP(\{\mathcal{Z}_0+\mathcal{Z}_\mathcal{L}> \alpha t\} \cap \mathcal{G}_0\cap\{\mathcal{Z}_{\ell+1}<L_{\ell+1}(\mathcal{Z}_0)\}).
\end{equation}
If one uses a deterministic barrier $L_{\ell+1}\approx \alpha(t_{\ell+1}-t_0)-\mathcal{C}\log_{\ell+1} t$ here (corresponding to the typical position of $\mathcal{Z}_0$) then the position of the process at $\ell+1$ is at most $L_{\ell+1}+\mathcal{Z}_0$. This position could be as large as $\approx \alpha t_{\ell+1}+\sqrt{t_0}+\mathcal{C}\log_{\ell+1}t$.  This additional term coming from the variance of the initial contribution does not permit a strong enough bound in the calculation below.

First conditioning on the value of $\mathcal{Z}_0$ and decomposing on the value at $t_{\ell+1}$ we find:
\begin{align*}
  \PP(\{\mathcal{Z}_\mathcal{L}-\mathcal{Z}_{\ell+1} &> \alpha t-\mathcal{Z}_{\ell+1}-w\} \cap \{\mathcal{Z}_{\ell+1}<L_{\ell+1}(w)\})\\
  &\leq\sum_{u\leq L_{\ell+1}(w)}\PP(\mathcal{Z}_{\mathcal{L}}-\mathcal{Z}_{\ell+1}>\alpha t-(u+1)-w)\PP(\mathcal{Z}_{\ell+1}\in(u,u+1]),
\end{align*}
using independence.  By the Gaussian decay, this is therefore bounded by
\begin{align*}
  &\ll \sum_{u\leq L_{\ell+1}(w)}e^{-\frac{(\alpha t-(u+1)-w)^2}{2(v_\mathcal{L}^2 - v_{\ell+1}^2)}}\frac{v_{\ell+1}}{u}e^{-\frac{u^2}{2v_{\ell+1}^2}}\\
  &\ll \frac{1}{\alpha\sqrt{t}}e^{-\alpha^2\frac{(t_\mathcal{L}-t_{\ell+1})^2}{2(v_\mathcal{L}^2 - v_{\ell+1}^2)}}e^{-\alpha^2\frac{t^2}{2v_{\ell+1}^2}}e^{2\alpha t\frac{\alpha(t_\mathcal{L}-t_{\ell+1})+w}{2v_{\ell+1}^2}}\sum_{u>\mathcal{C}\log_{\ell+1}t}e^{-2\alpha\frac{(u-1)(t_\mathcal{L}-t_{\ell+1})}{2(v_\mathcal{L}^2 - v_{\ell+1}^2)}}e^{-\frac{(u-1)^2}{2(v_\mathcal{L}^2 - v_{\ell+1}^2)}}e^{2\alpha t\frac{u}{2v_{\ell+1}^2}},
\end{align*}
achieved by shifting the sum by $(t_{\ell+1}-t_0)m(w)$ and expanding.  The pre-factor simplifies to
\[\frac{1}{\alpha\sqrt{t}} \exp\left(-\alpha^2 t\right)\cdot \exp\left((-\alpha^2 t_0 + 2\alpha w)-2\alpha^2\delta-\alpha^2\mathfrak{s}\log_{\mathcal{L}}t\right).\]
After unconditioning on the value of $\mathcal{Z}_0$, this term is like $\PP(\mathcal{Z}_0+\mathcal{Z}_\mathcal{L}>\alpha t)$. Similarly, after simplifying and estimating at the dominant $u$, the sum is bounded by
\begin{align*}
  &\ll\exp\left(2\alpha-4\alpha\mathcal{C}\delta\frac{\log_{\ell+1}t}{t_\mathcal{L}-t_{\ell+1}}(1-e^{t_{\ell+1}-t_\mathcal{L}})-\frac{(\mathcal{C}\log_{\ell+1}t-1)^2}{2(v_\mathcal{L}^2-v_{\ell+1}^2)}\right)
  \ll  c_{\alpha}\exp\left(-\frac{\mathcal{C}^2}{\mathfrak{s}}\log_{\ell+1}t\right).
\end{align*}
Again, the coefficient in the final exponent involving $\log_{\ell+1}t$ is strictly negative (and we implicitly use that $\mathcal{C}$ is strictly positive when estimating the sum above), resulting in a subdominant contribution for most $\ell$. For the largest\footnote{For $\ell=\mathcal{L}-1$, the probability involved is zero since the events $\{\mathcal{Z}_0+\mathcal{Z}_\mathcal{L}>\alpha t\}$ and $\mathcal{Z}_\mathcal{L}<L_\mathcal{L}(\mathcal{Z}_0)$ are incompatible.} $\ell$, namely $\mathcal{L}-2$, notice that $e\leq\log_{\ell+1}t<e^2$. The term $c_\alpha$ is an explicit coefficient dependent on $\alpha$,
\[c_\alpha=\exp\left(2\alpha-\frac{\mathcal{C}}{\mathfrak{s}}\left(4\alpha\delta + \mathcal{C}\log_\mathcal{L} t+2\mathcal{C}\delta -2\right)\right).\]
The coefficient will appear in the sum over $\ell$ below. Recall $\mathfrak{s}=35\log(\alpha+1)+\log(2\cdot 10^6)$ and $\delta>\delta^\star(\alpha)$  (cf.\;\eqref{eqn: choice of s} and \eqref{eqn: delta star} respectively). Therefore for small $\alpha$, both $\mathcal{C}/\mathfrak{s}, 2\delta>1$ so $c_\alpha<1/100$ say. For large $\alpha$, we have $\delta^\star/\mathfrak{s}\approx 1$ so again $c_\alpha<\exp(-10\alpha)$ say.

Therefore, returning to the original probability, unconditioning and comparing to \eqref{eqn: heuristic}, we have for $0\leq \ell\leq \mathcal{L}-2$
\begin{align*}
  \PP(\{\mathcal{Z}_0+\mathcal{Z}_\mathcal{L}&> \alpha t\} \cap \mathcal{G}_0\cap\{\mathcal{Z}_{\ell+1}<L_{\ell+1}(\mathcal{Z}_0)\})\ll \PP(\mathcal{Z}_0 + \mathcal{Z}_\mathcal{L} > \alpha t)\cdot  c_\alpha e^{-\frac{\mathcal{C}^2}{\mathfrak{s}}\log_{\ell+1}t}.
\end{align*}

Overall from \eqref{eq:prop2_expansion} (using the estimate \eqref{eqn: remark a} to remove $\mathcal G_0$ in the first instance) we find
\begin{align*}
  \PP(\mathcal{G}(V))&\gg \PP(\mathcal Z_0+\mathcal Z_\mathcal{L}>V)\left(1-\sum_{j=0}^{\mathcal{L}-1}e^{-(2\alpha\mathcal{B}-\alpha^2\mathfrak{s})\log_{\ell+1} t}-c_{\alpha}\sum_{j=0}^{\mathcal{L}-2} e^{-\frac{\mathcal{C}^2}{\mathfrak{s}}\log_{\ell+1}t}\right)\\
  &\asymp \PP(\mathcal Z_0+\mathcal Z_\mathcal{L}>V),
\end{align*}
where the final bound follows from the rapid convergence of the sum and taking $\mathcal{B}, \mathcal{C}$ as in \eqref{eqn: BC} (note that $c_\alpha$ is also small according to the discussion above) so that the term in brackets is strictly positive and can be taken independent of $\alpha$. 
\subsection{Proof of Proposition \ref{prop: S to G}} 
We first need two lemmas whose goal is to approximate quantitatively an indicator function by a polynomial. 
The content of these lemmas can be found in Lemma 3 in \cite{argbourad23}. We reproved them here for completeness, and also to get a clear dependence on the parameters, which is key in estimating the constant $C_k$. 
\begin{lem}
  \label{lem: h}
  Let $c>b>a>1$ and $\Delta$ large enough. There exist entire functions $h^+$ and $h^-$ depending on these parameters such that for all $x\in \R$,
  \begin{enumerate}[\normalfont(i)]
  \item\label{eqn: lem1 prop1} $0\leq h^-\leq h^+\leq 1,$
  \item\label{eqn: h+}  $
      \1(x\in [0,\Delta^{-1}])\Big(1-e^{-\Delta^{c-b-1}}\Big)\leq h^+(x)\leq \1(x\in [-\Delta^{-a},\Delta^{-1}+\Delta^{-a}]) + e^{-\Delta^{c-a-1}},
      $
  \item \label{eqn: h-}
      $\1(x\in [\Delta^{-a},\Delta^{-1}-\Delta^{-a}])\Big(1- e^{-\Delta^{c-a-1}}\Big)\leq h^-(x)\leq  \1(x\in [0,\Delta^{-1}])+ e^{-\Delta^{c-b-1}},$
  \item\label{eqn: lem1 prop4} for all $\ell\in \N$, the Fourier transform $\widehat{h^{\pm}}$ is such that
    \[
    \left|\int_\R \xi^\ell \widehat{h^\pm}(\xi) \rd \xi\right|
    \leq \Delta^{c(\ell+2)}.
    \]
  \end{enumerate}
\end{lem}

\begin{proof}
  We define $h^{\pm}$ as convolutions with a suitable kernel. Following Lemma 5 in \cite{argbourad20} based on a classical construction of \cite{ingham34}, we know there exists a real function $F$ such that $0\leq F\leq 1$ with decay
  \[
  F(x)\ll \exp(-|x|/\log^2(|x|+10)),
  \]
  for which the Fourier transform is supported on $[-1,1]$. With this in mind, consider the corresponding PDF $\varphi(x)=F(x)/\|F\|_1$, and the scaled version:
  \[
  \varphi_c(y)=\Delta^{c}\varphi(\Delta^c y).
  \]
  Since $c>1$, the scale of $\varphi_c$ is much smaller than the width of the interval $[0,\Delta^{-1}]$ to be approximated.
  The function $h^\pm$ are defined by the convolution of an indicator function with $\varphi_c$:
  \begin{equation}
    \label{eqn: G+}
    h^+(x)=\int \1(y\in [-\Delta^{-b},\Delta^{-1}+\Delta^{-b}]) \varphi_c(x-y)\rd y.
  \end{equation}
  and
  \begin{equation}
    \label{eqn: G-}
    h^-(x)=\int \1(y\in [\Delta^{-b},\Delta^{-1}-\Delta^{-b}]) \varphi_c(x-y)\rd y.
  \end{equation}
  The functions $h^{\pm}$ should be thought of as smoothings of the indicator function $1(x\in [0,\Delta^{-1}])$. The slight overspill $\Delta^{-b}$ must be added to properly bound the indicator function close to $0$ and $\Delta^{-1}$. 
  It is necessary to take $c>b>1$, as this ensures that the effective width of the kernel is much smaller than the overspill. 
  Moreover, to bound $h^+$ above (and $h^-$ below) by an indicator function, the interval has to be enlarged (or reduced) by $\Delta^{-a}$.
  For this, we need $b>a$ to ensure the indicator function in the definition of $h^+$ is bounded by $\1(x\in [-\Delta^{-a},\Delta^{-1}+\Delta^{-a}])$, and similarly for $h^-$.

  Property~\ref{eqn: lem1 prop1} is obvious from the definitions \eqref{eqn: G+} and \eqref{eqn: G-}.
  The first inequality in Property~\ref{eqn: h+} is also clear when $x\notin [0,\Delta^{-1}]$. The case $x\in [0,\Delta^{-1}]$ becomes
  \[
  1-e^{-\Delta^{c-b-1}}\leq h^+(x).
  \]
  To prove this, note that for any $0\leq x\leq \Delta^{-1}$,
  \[
  h^+(x)=\int_{-\Delta^{-1}-\Delta^{-b}}^{\Delta^{-b}} \varphi_c(u+x)\rd u\geq \int_{-\Delta^{-b}}^{\Delta^{-b}} \varphi_c(u)\rd u\geq 1-e^{-\Delta^{c-b-1}},
  \]
  where the last inequality follows by the decay of $\varphi$.
  The right-hand side inequality of Property~\ref{eqn: h+} is also clear when $x\in  [-\Delta^{-a},\Delta^{-1}+\Delta^{-a}])$. If $x<-\Delta^{-a}$, then we have
  \[
  h^+(x)=\int_{-\Delta^{-1}-\Delta^{-b}}^{\Delta^{-b}} \varphi_c(u+x)\rd u\leq \int_{-\infty}^{\Delta^{-b}-\Delta^{-a}}\varphi_c(u)\rd u\leq e^{-\Delta^{c-a-1}},
  \]
  for $\Delta$ large enough and by the decay of $\varphi$ (since $b>a$, we can ask for  $\Delta^{-b}-\Delta^{-a}\leq -\frac{1}{2}\Delta^{-a}$ say). Similarly, if $x>\Delta^{-1}+\Delta^{-a}$, we get
  \[
  h^+(x)=\int_{-\Delta^{-1}-\Delta^{-b}}^{\Delta^{-b}} \varphi_c(u+x)\rd u\leq \int_{\Delta^{-a}-\Delta^{-b}}^{\infty}\varphi_c(u)\rd u\leq e^{-\Delta^{c-a-1}}.
  \]
  Property~\ref{eqn: h-} is proved exactly as Property~\ref{eqn: h+} by using the definition of $h^-$ in \eqref{eqn: G-} instead.

  For Property~\ref{eqn: lem1 prop4}, note that 
  \[
  \widehat{h^{\pm}}(\xi)=\widehat \varphi_c(\xi) \cdot \widehat{\1_{[0,\Delta^{-1}]}} (\xi).
  \]
  On the one hand, by integration, we have the simple bound
  \[
  \widehat{\1_{[0,\Delta^{-1}}]} (\xi)\ll \Delta^{-1}.
  \]
  This smaller than $1$ for $\Delta$ large enough. On the other hand, since $\varphi$ is supported on $[-1,1]$, we have
  \[
  \left|\int \xi^\ell\widehat \varphi_c(\xi)\rd \xi\right|\leq \Delta^{c\ell}\int |\widehat{\varphi_c}(\xi)|\rd \xi.
  \]
  By the Cauchy-Schwarz inequality and Parseval's identity, we can bound the integral by
  \[
  \int|\widehat{\varphi_c}(\xi)|\rd \xi\leq \sqrt{2\Delta^c}\left(\int \varphi_c(x)^2 \rd x\right)^{1/2}
  \ll \Delta^c,
  \]
  since $\varphi_c(x)\ll \Delta^c$. This can be made $\leq \Delta^{2c}$ by taking $\Delta$ large enough to absorb the implicit constant.
  This gives the desired bound.
\end{proof}
In a second step, we approximate $h^\pm$ by polynomials.
\begin{lem}
  \label{lem: approx}
  Let $a>2$, $X>1$ and $\Delta$ large enough. There exist polynomials $\mathcal D^+$ and $\mathcal D^-$ of degree smaller than $100 X\Delta^{3a}$ with $\ell$-th coefficient bounded by $\frac{(2\pi)^\ell}{\ell!} \Delta^{2a(\ell+2)}$ such that for $|x|\leq X$ we have
  \begin{equation}
    \label{eqn: D+}
    \1(x\in [0,\Delta^{-1}])\Big(1-e^{-\Delta^{a-2}}\Big)\leq |\mathcal D^+(x)|^2\leq \1(x\in [-\Delta^{-a},\Delta^{-1}+\Delta^{-a}]) + e^{-\Delta^{a-2}},
  \end{equation}
  and
  \begin{equation}
    \label{eqn: D-}
    \1(x\in [\Delta^{-a},\Delta^{-1}-\Delta^{-a}])\Big(1- e^{-\Delta^{a-2}}\Big)\leq| \mathcal D^-(x)|^2\leq  \1(x\in [0,\Delta^{-1}])+ e^{-\Delta^{a-2}}.
  \end{equation}
\end{lem}

\begin{proof}
  We consider the entire functions $h^{\pm}$ from Lemma \ref{lem: h}. We pick $b=2a$ and $c=3a$. 
  Expressing $h^\pm$ in terms of their Fourier transform, we get
  \begin{align*}
    h^{\pm}(x)=\int e^{2\pi \ii \xi x}\widehat{h^{\pm}}(\xi) \rd \xi=\sum_{\ell\geq 0} \frac{(2\pi \ii x)^\ell}{\ell !}\int \xi^\ell \widehat{h^\pm} \rd \xi
    =\sum_{\ell<\nu}  \frac{(2\pi \ii x)^\ell}{\ell !}\int \xi^\ell \widehat{h^\pm}(\xi) \rd \xi +\mathcal R(x).
  \end{align*}
  The cutoff $\nu$ will need to be suitably picked. By Taylor's theorem, the remainder term satisfies the bound
  \begin{align*}
    |\mathcal R|\leq \frac{(2\pi X)^\nu}{\nu !} \Big|\int \xi^{\nu} \widehat{h^\pm}(\xi) \rd \xi\Big|\leq \frac{(2\pi X)^\nu}{\nu!} \Delta^{c(\nu+2)},
  \end{align*}
  where we used Property~\ref{eqn: lem1 prop4} of Lemma \ref{lem: h}. Using Stirling's formula, we see that  by picking
  \begin{equation}
    \label{eqn: nu}
    \nu=100X \Delta^c,
  \end{equation}
  we have
  \begin{equation}
    \label{eqn: R bound}
    |\mathcal R|\leq e^{-100 X\Delta^c}.
  \end{equation}
  Thus far we have established the following decomposition
  \[
  h^{\pm}(x)=\mathcal D^\pm(x) +\mathcal R(x),
  \]
  where $\mathcal D^{\pm}$ is a polynomial of degree at most $100X \Delta^{c}$.
  The $\ell$-th coefficient of $\mathcal D^{\pm}$, $\ell<\nu$, satisfies the bound
  \[
  \left| \frac{(2\pi \ii)^\ell}{\ell !}\int \xi^\ell \widehat{h^{\pm}}(\xi) \rd \xi \right|
  \leq \frac{(2\pi )^\ell}{\ell!}\Delta^{c(\ell+2)},
  \]
  by Property~\ref{eqn: lem1 prop4} of Lemma \ref{lem: h}.

  It remains to prove the inequalities \eqref{eqn: D+} and \eqref{eqn: D-}. 
  We prove the left-hand side inequality of \eqref{eqn: D+}, the other ones are proved the same way. The case $x\notin [0,\Delta^{-1}]$ is clear. When $x\in  [0,\Delta^{-1}]$,
  we have from Property~\ref{eqn: h+} of Lemma~\ref{lem: h} and \eqref{eqn: R bound} that
  \[
  | \mathcal D^+(x)|=|h^{+}(x)-\mathcal R(x)| \geq 1- e^{-\Delta^{2a-1}}-e^{-100 X \Delta^c}.
  \]
  This is  $1- e^{-\Delta^{a-2}}$ for $\Delta$ large enough.
\end{proof}
We now apply the last two lemmas to the proof of Proposition  \ref{prop: S to G}.
\begin{proof}[Proof of Proposition \ref{prop: S to G}]
  For simplicity, we prove the case $\ell=\mathcal L$. The proof when $\ell<\mathcal L$ is the same.
  Also, without loss of generality, we can assume $\alpha\geq 1$ by replacing $\alpha$ by $\alpha\vee 1$ in what follows.  

  Define the increments
  \[
  Y_\ell=S_{\ell}- S_{\ell-1}, \ 2\leq \ell\leq \mathcal L,
  \]
  and $Y_1=S_1$ for $\ell=1$. For $\ell=0$, we set
  \begin{equation}
    \label{eqn: Y0}
    Y_0=|\mathcal M_0|^2.
  \end{equation}
  The technical reason behind this is that $\mathcal M_0$ is a short Dirichlet polynomial, but $|\mathcal M_0^{-1}|$ and  $\log |\mathcal M_0^{-1}|$ are not. 
  Therefore, it is technically easier to approximate events involving $|\mathcal M_0|^2$.

  Consider the events $G_\ell$ as in \eqref{eqn: G} and $G(V)$ as in \eqref{eqn: G0}.
  We want to express these in terms of the increments
  \[
  \{Y_\ell \in [u_\ell, u_\ell+\Delta_{\ell}^{-1}]\}, \ 0\leq \ell\leq \mathcal L.
  \]
  The set $\mathcal I=\mathcal I(V)$ of tuples $\mathbf u=(u_\ell,0\leq \ell\leq \mathcal L)$ with $u_\ell\in \Delta_\ell^{-1}\cap \mathbb Z$ 
  must be picked to mimic the restrictions in the $G_\ell$'s and $G(V)$. 
  For $\ell=0$, from the definition of $G_0$ in \eqref{eqn: G0}, we take $u_0\in [L_0,U_0]$. Clearly, $|u_0|\leq 1$. We discretize this interval with a step size of
  \begin{equation}
    \label{eqn: delta0}
    \Delta_0^{-1}=e^{-100\alpha t_0}.
  \end{equation}
  Note that for $Y_0\in [u_0,u_0+\Delta_0^{-1}]$, we then have by expanding the logarithm
  \[
  -\frac{1}{2}\log u_0-\frac{1}{2}\Delta_0^{-1/2}\leq-\frac{1}{2}\log Y_0\leq- \frac{1}{2}\log u_0
  \]
  For the $u_\ell$, $1\leq \ell\leq \mathcal L$, and $s_0=-\frac{1}{2}\log u_0$, we ask that the $u$'s satisfy the ``barrier" condition,
  \begin{equation}
    \label{eqn: barrier u}
    L_\ell(s_0)\leq s_0+ \sum_{1\leq k\leq \ell} u_k\leq U_\ell(s_0).
  \end{equation}
  In the preceding and following, we write $L_\ell=L_\ell(s_0)$. The condition \eqref{eqn: barrier u} implies that for $1\leq \ell\leq \mathcal L$:
  \[
  L_\ell - U_{\ell-1} \leq u_\ell \leq U_\ell - L_{\ell-1}.
  \]
For $\ell\geq 2$ this yields the bounds
  \begin{align*}
       u_\ell&\leq U_\ell-L_{\ell-1}=(\mathfrak{s}\,m(s_0)+\mathcal{C})\log_{\ell-1}t -(\mathfrak{s}\,m(s_0)-\mathcal{B})\log_{\ell}t,\\
    u_\ell&\geq L_{\ell}-U_{\ell-1}=(\mathfrak{s}\,m(s_0)-\mathcal{B})\log_{\ell-1}t -(\mathfrak{s}\,m(s_0)+\mathcal{C})\log_{\ell}t.
  \end{align*}
  From the choice of slope $m(z)$ in \eqref{eqn: slope} and $\mathcal B$ and $\mathcal C$ in \eqref{eqn: BC}, we can assume that
  \begin{equation}
    \label{eqn: bound u}
    |u_\ell|\leq 2 \alpha \mathfrak s \log_{\ell-1}t, \quad 2\leq \ell\leq \mathcal L
  \end{equation}
  and $|u_1|\leq 2\alpha t_1$.
  This is the important input for the choice of parameters. 
  From this, we take
  \[
  \Delta_\ell=10 \alpha \mathfrak s \log_{\ell-1}t, \quad 2\leq \ell\leq \mathcal L;\qquad \Delta_\ell = 10\alpha t_1.
  \]
  This ensures that
  \begin{equation}
    \label{eqn: Delta summable}
    \sum_{1\leq k\leq \ell}\Delta_k^{-1}\leq \alpha^{-1}, \quad 1\leq \ell\leq \mathcal L.
  \end{equation}
  This choice is also motivated by the condition $|u_\ell\Delta_\ell^{-1}|\leq 1$ arising later in \eqref{eqn: overspill u}.
  Finally, the restriction $\{\mathcal S_\mathcal L+\log |\mathcal M_0^{-1}|>V\}$ imposes the following condition on $\mathcal I$ 
  \begin{equation}
    \label{eqn: IV}
    \sum_{1\leq\ell\leq \mathcal L} u_\ell -\frac{1}{2}\log u_0>V\ .
  \end{equation}

  With these choices, it is not hard to check that we have the following inclusions:
  \begin{align*}
    & \bigcup_{\mathbf{u} \in \mathcal I}\{Y_\ell \in [u_\ell, u_\ell+\Delta_\ell^{-1}], \ell\leq \mathcal L\}
    \subset \\
    &
    \{S_\mathcal L-\frac{1}{2}\log Y_0>V-\alpha^{-1}, S_\ell\in [L_\ell, U_\ell+\alpha^{-1}], 1\leq \ell\leq \mathcal L, Y_0\in [L_0, U_0+\Delta_0^{-1}]\}\ ,\numberthis\label{eqn: inclusion}
  \end{align*}
  and
  \begin{align*}
    & \bigcup_{\mathbf{u} \in \mathcal I}\{Y_\ell \in [u_\ell, u_\ell+\Delta_\ell^{-1}],\ell\leq \mathcal L\}
    \supset \\
    &
    \{S_\mathcal L-\frac{1}{2}\log Y_0>V+\alpha^{-1}, S_\ell\in [L_\ell+\alpha^{-1}, U_\ell], 1\leq \ell\leq \mathcal L, Y_0\in [L_0+\Delta_0^{-1}, U_0]\}\ .\numberthis    \label{eqn: inclusion2}
  \end{align*}
  Note that the events on the left are disjoint for two different tuples $\mathbf u$.

  The lower and upper bound in the statement of Proposition \ref{prop: S to G} are proved separately using the two above inclusions.
  Both proofs consist in re-expressing the probabilities of the increments $Y_k$ in terms of the ones of a random model where we effectively replace $p^{-i\tau}$ by a uniform random variable on the unit circle.
  This will be achieved by converting the indicator functions in terms of polynomials with Lemma \ref{lem: approx} and using mean-value theorems for Dirichlet polynomials.
  The random model is defined as follows: for $(\theta_p, p \text{ primes})$ be IID random variables uniformly distributed on $[0,2\pi]$, we consider the random model for the Dirichlet sums $S_\ell$
  \begin{equation}
    \label{eqn: Sl random}
    \mathcal S_\ell=\sum_{T_0< p \leq \exp(e^{t_\ell})} \frac{\re \ e^{\ii \theta_p}}{p^{\sigma}}+\frac{\re \ e^{2\ii \theta_p}}{2p^{2\sigma}},\quad 1\leq \ell\leq \mathcal L,
  \end{equation}
  and their corresponding increments
  \begin{equation}
    \label{eqn: Yl}
    \mathcal Y_\ell=\mathcal S_\ell- \mathcal S_{\ell-1}, \ 2\leq \ell\leq \mathcal L,
  \end{equation}
  with $\mathcal Y_1=\mathcal S_1$.
  For $\ell=0$, keeping in line with the choice in \eqref{eqn: Y0}, we define
  \[
  \mathcal Y_0=\prod_{p\leq T_0}\Big|1-\frac{e^{\ii \theta_p}}{p^{\sigma}}\Big|^2.
  \]

  \noindent{\it Proof of $\PP(G(V))\gg \PP(\mathcal Z_0+\mathcal Z_\mathcal L>V)$.}\\
  It suffices to prove that  
  \begin{equation}
    \label{eqn: G to mathcal G}
    \PP(G(V))\gg \PP(\mathcal G(V+\alpha^{-1}))
  \end{equation}
  for the event $\mathcal G(V)$ given in \eqref{eqn: G(V)}. Indeed, note that the added terms $+\alpha^{-1}$ to the barrier $U_\ell$ and $L_\ell$ are irrelevant by Proposition \ref{prop: G}.
  Moreover, the same proposition implies that $\PP(\mathcal G(V+\alpha^{-1}))\gg  \PP(\mathcal Z_0+\mathcal Z_\mathcal L>V+\alpha^{-1})$. This is $\gg \PP(\mathcal Z_0+\mathcal Z_\mathcal L>V)$ by \eqref{eqn: heuristic}.

  To prove \eqref{eqn: G to mathcal G}, we use the inclusion \eqref{eqn: inclusion} to get
  \[
  \PP(G(V))\gg\sum_{\mathbf u \in \mathcal I}\PP(Y_\ell \in [u_\ell, u_\ell+\Delta_\ell^{-1}], 0\leq \ell\leq \mathcal L).
  \]
  To apply Lemma \ref{lem: approx} to the indicator function $[0,\Delta_\ell^{-1}]$, we introduce the restriction $|Y_\ell-u_\ell|\leq X_\ell$ for a parameter $X_\ell\leq 100\Delta_\ell^{4a}$. 
  The justification of the constraint will be specified in \eqref{eqn: X}. We will also need to pick $a=5$, to guarantee that $\Delta_\ell^{a-2}>\Delta_\ell^{2}$. Together with \eqref{eqn: choice of s}, this informs the choice~\eqref{eqn: s}.
  With this, we have for $1\leq \ell \leq \mathcal L$
  \begin{equation}
  \label{eqn: fix}
  \begin{aligned}
    \1(Y_\ell \in [u_\ell, u_\ell+\Delta_\ell^{-1}])&\geq  \1_{[0,\Delta_\ell^{-1}]}(Y_\ell-u_\ell)\1(|Y_\ell-u_\ell|\leq X_\ell)\\
    &\geq \left(|\mathcal D_\ell^{-}(Y_\ell-u_\ell)|^2-e^{-\Delta_\ell^{a-2}}\right) \1(|Y_\ell-u_\ell|\leq X_\ell),
    \end{aligned}
 \end{equation}
  where $\mathcal D_\ell^{-}(x)$ is a polynomial of degree $\leq 100 X_\ell \Delta_\ell^{3a}$ as guaranteed by Lemma \ref{lem: approx}. 
  For $\ell=0$, there is no need to introduce the indicator function. Indeed, we have the pointwise bound
  \[
  |Y_0|=\prod_{p\leq T_0}\Big|1-\frac{p^{-\ii \tau}}{p^\sigma}\Big|^2\leq \exp(10T_0^{1/2}).
  \]
  Therefore, by picking $X_0=\exp(10T_0^{1/2})=\exp(10\sqrt{\log\log T})$, we get
  \begin{equation}
  \label{eqn: fix2}
  \1(Y_0 \in [u_0, u_0+\Delta_0^{-1}])\geq
|\mathcal D_0^{-}(Y_0-u_0)|^2-e^{-\Delta_0^{a-2}}.
\end{equation}
  This means that $\prod_{0\leq\ell\leq \mathcal L}\mathcal D_\ell^{-}(Y_\ell-u_\ell)$ is a Dirichlet polynomial of length less than 
  \begin{equation}
    \label{eqn: length D}
    \exp(2 \cdot 100e^{t_\mathcal L}X_\mathcal L\Delta_\mathcal L^{3a})=\exp\Big(2\cdot 10^{4}\frac{(\alpha\vee 1)^{7a}(\mathfrak s \log_\mathcal L t)^{7a}}{e^{\mathfrak s \log _\mathcal L t}}\Big).
  \end{equation}
  To ensure that the length is less than $T^{1/100}$, we are led to impose
  \begin{equation}
    \label{eqn: choice of s}
    \boxed{
      \mathfrak s\log_\mathcal Lt-7a\log (\mathfrak s \log_\mathcal L t)> 7a\log (\alpha\vee 1)+\log(2 \cdot10^6).
    }
  \end{equation}
  For $\alpha$ large, this essentially imposes $\mathfrak s \log_{\mathcal L}t$ larger than $7a\log (\alpha\vee 1)$.
  Our goal is now to get rid of the indicator function for $1\leq \ell\leq \mathcal L$.
  In this case, writing $\1(|Y_\ell-u_\ell|\leq X_\ell)=1-\1(|Y_\ell-u_\ell|>X_\ell)$, we get by expanding the product
  \begin{align*}
    &\E\left[\prod_{0\leq \ell\leq \mathcal L} \left(|\mathcal D_\ell^{-}(Y_\ell-u_\ell)|^2-e^{-\Delta_\ell^{a-2}}\right)\prod_{1\leq \ell\leq \mathcal L}\1(|Y_\ell-u_\ell|\leq X_\ell)\right]\\
    &=\sum_{J\subseteq \{1,\dots, \mathcal L\}}(-1)^{|J|}\E\left[\prod_{0\leq \ell\leq\mathcal L}\left(|\mathcal D_\ell^{-}(Y_\ell-u_\ell)|^2-e^{-\Delta_\ell^{a-2}}\right)\prod_{j\in J}\1(|Y_j-u_j|> X_j)\right].\numberthis\label{eqn: D expansion}
  \end{align*}
  The dominant term will be the summand corresponding to $J=\emptyset$. 
  It is good to start with estimating this term.
  By the mean-value theorem of Dirichlet polynomials, see Lemma \ref{lem: MV}, the first term can be expressed in terms of the random increments $\mathcal Y_\ell$, which are now independent:
  \begin{equation}
    \label{eqn: dominant con}
  \E\left[\prod_\ell \left(|\mathcal D_\ell^{-}(Y_\ell-u_\ell)|^2-e^{-\Delta_\ell^{a-2}}\right)\right]=(1+\OO(T^{-1/100}))\prod_\ell\E\left[\left(|\mathcal D_\ell^{-}(\mathcal Y_\ell-u_\ell)|^2-e^{-\Delta_\ell^{a-2}}\right)\right].
  \end{equation}
  We can apply Lemma \ref{lem: approx} again to convert the polynomial back to an indicator function
  \begin{align*}
    \E\left[|\mathcal D_\ell^{-}(\mathcal Y_\ell-u_\ell)|^2\right]&\geq \E\left[|\mathcal D_\ell^{-}(\mathcal Y_\ell-u_\ell)|^2\1(|\mathcal Y_\ell-u_\ell|\leq X_\ell)\right]\\
    &\geq(1-e^{-\Delta_\ell^{a-2}})\PP(\mathcal Y_\ell-u_\ell\in [\Delta_\ell^{-a},\Delta_\ell^{-1}-\Delta_\ell^{-a}]),
  \end{align*}
  where we noticed that $\mathcal Y_\ell-u_\ell\in [\Delta_\ell^{-a},\Delta_\ell^{-1}-\Delta_\ell^{-a}]$ implies $|\mathcal Y_\ell-u_\ell|\leq X_\ell$.
  The goal is to rewrite the probability in terms of Gaussian increments and without the overspill $\Delta_\ell^{-a}$.
  This is evaluated differently for $\ell=0$, $\ell=1$ and $1<\ell\leq \mathcal L$.

  For the range $1<\ell\leq \mathcal L$, by Lemma \ref{lem: gaussian approx}, the probability is
  \begin{equation}
    \label{eqn: gaussian approx}
    \PP( \mathcal N_\ell-u_\ell\in [\Delta_\ell^{-a},\Delta_\ell^{-1}-\Delta_\ell^{-a}])+e^{-ce^{t_\ell/2}},
  \end{equation}
  for the Gaussian increment $\mathcal N_\ell=\mathcal Z_{\ell}-\mathcal Z_{\ell-1}$. 
  To get rid of the overspills $\Delta_\ell^{-a}$, note that by the choice of $u_\ell$ and $\Delta_\ell$, we have 
  \begin{equation}
    \label{eqn: overspill u}
    |u_\ell|\Delta_\ell^{-a}\leq 1,
  \end{equation}
  so that
  \begin{equation}
    \label{eqn: overspill out}
    \PP( \mathcal N_\ell-u_\ell\in [\Delta_\ell^{-a},\Delta_\ell^{-1}-\Delta_\ell^{-a}])=(1+\OO(\Delta_\ell^{-a}))\PP( \mathcal N_\ell-u_\ell\in [0,\Delta_\ell^{-1}]).
  \end{equation}
  Moreover, we have trivially 
  \begin{equation}
    \PP( \mathcal N_\ell-u_\ell\in [0,\Delta_\ell^{-1}])\gg \Delta_\ell^{-1} e^{-u_\ell^2/2v_\ell}\gg e^{-\Delta_\ell^2},
  \end{equation}
  by the choice of $\Delta_\ell$.
  This is much larger than the additive error in \eqref{eqn: gaussian approx}.
  We conclude that for each $1<\ell\leq \mathcal L$
  \begin{equation}
    \label{eqn: 2 inequalities}
    \E\left[|\mathcal D_\ell^{-}(\mathcal Y_\ell-u_\ell)|^2\right]\geq  (1+\OO(\Delta_\ell^{-a})) \PP( \mathcal N_\ell\in [u_\ell,u_\ell+\Delta_\ell^{-1}])\gg e^{-\Delta_\ell^2}.
  \end{equation}
  Putting this together with the choice $a=5$ allows us to rewrite \eqref{eqn: dominant con} as
  \begin{align*}
    \E\left[\left(|\mathcal D_\ell^{-}(\mathcal Y_\ell-u_\ell)|^2-e^{-\Delta_\ell^{a-2}}\right)\right]&\geq \Big(1-e^{-\Delta_\ell^{a-2}+\Delta_\ell^2}\Big)\E\left[|\mathcal D_\ell^{-}(\mathcal Y_\ell-u_\ell)|^2\right]\\
    &\geq (1+\OO(\Delta_\ell^{-a})) \PP( \mathcal N_\ell\in [u_\ell,u_\ell+\Delta_\ell^{-1}]).\numberthis \label{dominant DY}
  \end{align*}
  
  For $\ell=1$, the additive error in Lemma \ref{lem: gaussian approx} is too big compared to the probability. 
  Taking a larger $t_0$ to reduce the error would cause the polynomial $\mathcal M_0$ to be too long.
  To circumvent this issue, we estimate the probability using Lemma \ref{lem: gaussian approx 1}.
  This is done at the cost of a multiplicative constant, which is harmless when it appears for a single term in the product over $\ell$.
  We ultimately get 
  \begin{equation}
    \label{eqn: l=1}
\E\left[|\mathcal D_1^{-}(\mathcal Y_1-u_1)|^2-e^{-\Delta_1^{a-2}}\right]
      \gg \PP( \mathcal N_1\in [u_1,u_1+\Delta_1^{-1}])
  \end{equation}
where $\mathcal{N}_1$ is a centred Gaussian with variance $v_1^2$.
  We now establish a similar bound for $\mathcal Y_0$. 
  Lemma \ref{lem: approx} implies that
  \[
  \E\left[|\mathcal D_0^{-}(\mathcal Y_0-u_0)|^2\right]\gg 
  \PP(\mathcal Y_0\in [u_0+\Delta_0^{-a},u_0+\Delta_0^{-1}- \Delta_0^{-a}])
  \geq \frac{1}{2} \PP(\mathcal Y_0\in [u_0+\Delta_0^{-1},u_0+2\Delta_0^{-1}]).
  \]
  To bound the probability from below, note that
  \[
  \E[\mathcal Y_0]\geq 1,\qquad \E[\mathcal Y_0^2]=\E[|\mathcal M_0|^4]=\prod_{p\leq T_0}\Big(1+\frac{2}{p^{2\sigma}}+\OO(p^{-3\sigma})\Big)\ll \exp(2t_0).
  \]
  This implies 
  \begin{equation}
    \label{eqn: Y0 estimate pre}
    \PP(\mathcal Y_0\in [u_0+\Delta_0^{-1},u_0+2\Delta_0^{-1}])
    \gg \Delta_0^{-1}\PP(\mathcal Y_0>u_0+2\Delta_0^{-1})
    \gg  \Delta_0^{-1} \frac{\E[\mathcal Y_0]^2}{\E[\mathcal Y_0^2]}\gg e^{-\Delta_0^{a-2}}.
  \end{equation}
  This ensures that in \eqref{eqn: dominant con}
  \begin{equation}
    \label{eqn: Y0 estimate}
    \E\left[|\mathcal D_0^{-}(\mathcal Y_0-u_0)|^2-e^{-\Delta_0^{a-2}}\right]\gg \E\left[|\mathcal D_0^{-}(\mathcal Y_0-u_0)|^2\right]\gg 
    \PP(\mathcal Y_0\in [u_0+\Delta_0^{-1},u_0+2\Delta_0^{-1}]).
  \end{equation}
  Altogether, Equations \eqref{dominant DY}, \eqref{eqn: l=1} and \eqref{eqn: Y0 estimate} show that the term for $J=\emptyset$ in \eqref{eqn: D expansion} is
  \begin{equation}
    \label{eqn: final estimate prop 1}
    \gg  \PP(\mathcal Y_0\in [u_0+\Delta_0^{-1},u_0+2\Delta_0^{-1}])\prod_{1\leq \ell\leq \mathcal L}\PP( \mathcal N_\ell\in [u_\ell,u_\ell+\Delta_\ell^{-1}]).
  \end{equation}

  It remains to estimate the summands in \eqref{eqn: D expansion} with $J\neq \emptyset$. 
  For each $J$, we get that each summand is bounded by (using that $\1\{a<b\}\leq (b/a)^{2q}$)
  \begin{equation}
    \label{eqn: estimate remainder}
    \leq \E\left[\prod_{0\leq\ell\leq\mathcal L}(d_\ell(Y_\ell) - \e_\ell)\prod_{j\in J}X_j^{-2q_j}|Y_j-u_j|^{2q_j}\right],
  \end{equation}
  for an appropriate choice of $q_j$, $j\in J$, and writing $d_\ell(y) = |D_\ell^{-}(y-u_\ell)|^2$, $\e_\ell = e^{-\Delta_\ell^{a-2}}$. We will show that this is subdominant to the contribution from $J=\emptyset$. Expanding the first factor we get
  \begin{equation}
    \label{eqn: expand mixed}
    \sum_{K\subseteq\{0,\dots,\mathcal{L}\}}\prod_{k\in K^c}\e_k \mathbb{E}\left[\prod_{\ell\in K}d_\ell(Y_\ell)  \prod_{j\in J}X_j^{-2q_j}|Y_j-u_j|^{2q_j}\right].
  \end{equation}
  The length of the polynomial in the expectation is then at most
  \begin{equation}
    \label{eqn: length DY}
    \exp\left(\sum_{0\leq \ell\leq \mathcal L}2 \cdot 10^2e^{t_\ell}X_\ell\Delta_\ell^{3a}+\sum_{j\leq \mathcal L}2e^{t_j}(2q_j)\right).
  \end{equation}
  The mean-value theorem of Dirichlet polynomials can be applied if $q_j$ is not too large, for example $q_j\leq X_j \Delta_j$ suffices. 
  Passing to the random model, Equation \eqref{eqn: estimate remainder} is
  \begin{equation}
    \label{eqn: random markov}
    \leq\sum_{K\subseteq\{0,\dots,\mathcal{L}\}}\prod_{k\in K^c}\e_k\prod_{j\in J}X_j^{-2q_j}\prod_{\ell\in K\cap J^c}\mathbb{E}\left[d_\ell(\mathcal{Y}_\ell)\right]\prod_{j\in K\cap J}\mathbb{E}\left[d_j(\mathcal{Y}_j)|\mathcal{Y}_j-u_j|^{2q_\ell}\right]\prod_{j\in J\cap K^c}\mathbb{E}\left[|\mathcal{Y}_j-u_j|^{2q_j}\right].
  \end{equation}
  For the term with $j\in K\cap J$, we may split the expectation using
  \begin{equation}
    \label{eqn: both terms}
    \ll \mathbb{E}\left[|\mathcal{D}_j^-(\mathcal{Y}_j - u_j)|^4\right]^{1/2}\mathbb{E}\left[\mathcal{Y}_j-u_j|^{4q_j}\right]^{1/2}.
  \end{equation}
  We now estimate the moments of $\mathcal{D}_j^{-}(\mathcal{Y}_j-u_j)$ and $|\mathcal{Y}_j-u_j|$ separately.  The fourth moment for the former, for $j\geq 1$, is estimated using the bounds on the coefficients in Lemma \ref{lem: approx} and the elementary estimates on the random model, see Equation \eqref{eqn: bound mgf gaussian} in Lemma \ref{lem: bound moment gaussian},
  \begin{align*}
    \E[|\mathcal D_j^{-}(\mathcal Y_j-u_j)|^4]&\leq \Delta_j^{4a}\E\left[\Big(\sum_{\ell\geq0}\frac{(2\pi\Delta_j^{2a})^\ell}{\ell!}|\mathcal Y_j-u_j|^\ell\Big)^4\right]\\
    &\leq \Delta_j^{4a}\E[\exp(8\pi\Delta_j^{2a}|\mathcal Y_j-u_j|)]\\
    &\leq \Delta_j^{4a}e^{8\pi \Delta_j^{2a}|u_j|+8^2\pi^2\Delta_j^{4a}v_j^2/2}\\
    &\leq \exp(100\Delta_j^{4a}),
  \end{align*}
  since $|u_j|\leq \Delta_j$.
  For the second factor in \eqref{eqn: both terms}, Minkowski's inequality implies
  \[\E[|\mathcal Y_j-u_j|^{4q_j}]^{1/(4q_j)}\leq \E[|\mathcal Y_j|^{4q_j}]^{1/(4q_j)}+|u_j|.\]
  Moreover, the bound \eqref{eqn: bound moment gaussian} yields $\E[|\mathcal Y_j|^{4q_j}]\leq \frac{(4q_j)!}{2^{2q_j} (2q_j)!}v_j^{4q_j}$, which by Stirling gives the estimate
  \[
  \E[|\mathcal Y_j|^{4q_j}]^{1/(4q_j)}\leq q_j^{1/2}.
  \]
  We are now in position to pick $q_j=X_j$. Note that this satisfies the restriction $q_j\leq X_j\Delta_j^{3a}$ imposed by \eqref{eqn: length DY}.
  With this choice, we have
  \[
  X^{-4q_j}\E[| \mathcal Y_j-u_j|^{4q_j}]\leq\Big(\frac{1}{X_j^{1/2}}+\frac{|u_j|}{X_j}\Big)^{4q_j}\leq e^{-4X_j}.
  \]
  Putting all these observations back in Equation \eqref{eqn: random markov} gives the upper bound
   $$
   \begin{aligned}
      &\leq \sum_{K\subseteq\{0,\dots,\mathcal{L}\}}\prod_{k\in K^c}\e_k\prod_{j\in J}e^{50\Delta_j^{4a}-2X_j}\prod_{j\in K\cap J^c}\mathbb{E}\left[d_j(\mathcal{Y}_j)\right]\\
       &=\Big(\prod_{0\leq\ell\leq\mathcal{L}}\mathbb{E}\left[d_\ell(\mathcal{Y}_\ell)\right]\Big)\sum_{K\subseteq\{0,\dots,\mathcal{L}\}}
      \prod_{j\in K\cap J} \frac{e^{50\Delta_j^{4a}-2X_j}}{\mathbb{E}[d_j(\mathcal{Y}_j)]} \prod_{j\in K^c\cap J}\e_j\cdot \frac{e^{50\Delta_j^{4a}-2X_j}}{\mathbb{E}[d_j(\mathcal{Y}_j)]} \prod_{j\in K^c\cap J^c} \frac{\e_j}{\mathbb{E}\left[d_j(\mathcal{Y}_j)\right]}\\
          &=\Big(\prod_{0\leq\ell\leq\mathcal{L}}\mathbb{E}\left[d_\ell(\mathcal{Y}_\ell)\right]\Big)  \prod_{j\in J} \frac{e^{50\Delta_j^{4a}-2X_j}}{\mathbb{E}[d_j(\mathcal{Y}_j)]} \Big(\sum_{K\subseteq\{0,\dots,\mathcal{L}\}}
     \prod_{j\in K^c\cap J}\e_j \prod_{j\in K^c\cap J^c} \frac{\e_j}{\mathbb{E}\left[d_j(\mathcal{Y}_j)\right]}\Big)\\
   &  \ll\Big(\prod_{0\leq\ell\leq\mathcal{L}}\mathbb{E}\left[d_\ell(\mathcal{Y}_\ell)\right]\Big)  \prod_{j\in J} \frac{e^{50\Delta_j^{4a}-2X_j}}{\mathbb{E}[d_j(\mathcal{Y}_j)]},
      \end{aligned}
      $$
where we used the bound \eqref{eqn: 2 inequalities} and the value of $\e_j$ in the last inequality.
  This leads to the choice 
  \begin{equation}
    \label{eqn: X}
    X_j=100\Delta_j^{4a}.
  \end{equation}
This shows that the terms with $J\neq \emptyset$ are indeed subdominant to \eqref{eqn: final estimate prop 1}. 
  The claim \eqref{eqn: G to mathcal G} follows after summing over ${\mathbf u}\in\mathcal{I}$ and appealing to the inclusion \eqref{eqn: inclusion} with $\mathcal Y_0$ and $\mathcal Z_\ell$. 
  
  \medskip
  \noindent{\it Proof of $\PP(G(V))\ll \PP(\mathcal Z_0+\mathcal Z_\mathcal L>V)$.}
  The proof of the upper bound follows the same method using the polynomial $\mathcal D^+$ instead of $\mathcal D^-$. The proof is in fact slightly easier as we do not need to bound a long polynomial as in \eqref{eqn: estimate remainder}. The corresponding quantity needs to be estimated in the random model where the factors are decoupled. 
\end{proof}

We will need the following spinoff of Proposition \ref{prop: G} in the proof of Proposition \ref{prop: M to S}. 
\begin{cor}
  \label{cor: decoupling gaussian}
  Consider the complex partial sums $\wS_k=\sum_{T_0<p \leq T_k} \frac{ p^{-\ii \tau}}{p^{\sigma}}+\frac{ p^{-2\ii\tau}}{2p^{2\sigma}}$, $1\leq \ell\leq \mathcal L$. For $1<\ell\leq \mathcal L$ and $q\leq 100\alpha^2(t_\ell-t_{\ell-1})$, we have
  \[
  \E[|\wS_\ell-\wS_{\ell-1}|^{2q}\1(G_{\ell-1})]\ll \E[|\wS_\ell-\wS_{\ell-1}|^{2q}]\PP(\{\mathcal Z_{\ell-1} \geq  L_{\ell-1}-(\alpha\vee 1)^{-1}\}\cap \mathcal G_0),
  \]
\end{cor}
\begin{proof}
  This is done the same way as the proof of Proposition \ref{prop: S to G} by approximating the indicator function $\1(G_{\ell-1})$ by polynomials. 
  The mean-value theorem for Dirichlet polynomial can be applied with the additional $(\wS_\ell-\wS_{\ell-1})^q$ since the total length is still admissible, as in Equations \eqref{eqn: length D} and \eqref{eqn: length DY}, 
  by the bound on $q$. Since we only need an upper bound, one can ultimately drop all restrictions on $\mathcal Z_\ell$, $1\leq \ell\leq \mathcal L$, except for the lower bound on $\mathcal Z_{\ell-1}$. The restriction on $\mathcal Z_0$ also remains.
\end{proof}

\subsection{Proof of Proposition \ref{prop: zeta to M}}
The main tool to prove the proposition is Lemma \ref{lem:mollification} from the appendix.
The following lemma adapts it to our setting.
\begin{lem}
  \label{lem: mollif adapted}
  Consider the mollifier $\mathcal M$ defined in \eqref{eqn: M total}.
  For $\alpha>0$, let $G(V)$ be the event in \eqref{eqn: G(V)} for $V\sim \alpha \log\log T$ and $\mathcal Z_\ell$, $0\leq \ell\leq \mathcal L$, the random variables defined in \eqref{eqn: Z0} and \eqref{eqn: Z_l}.
  We have for $\delta>\delta^\star(\alpha)$ given in \eqref{eqn: delta star}
  \[
  \E\left[\Big|\zeta\mathcal M -1\Big|^2\1(G(V))\right]\ll
  \left(\frac{e^{-2\delta}}{\delta}+e^{-\mu_\mathcal L/10}\right)\cdot \PP(\mathcal Z_0+\mathcal Z_\mathcal L>V).
  \]
\end{lem}

\begin{proof}
  The result follows from Lemma \ref{lem:mollification} after properly approximating the indicator function.
  To do this, we proceed as in the proof of Proposition \ref{prop: S to G} by writing $\1(G_\mathcal L)$ as a Dirichlet polynomial. 
  Writing the event in terms of the increments $Y_\ell$ and using the inclusion \eqref{eqn: inclusion} yields
  \begin{align*}
    &\sum_{\mathbf u \in \mathcal I}\E\left[\Big|\zeta\mathcal M -1\Big|^2\1(Y_\ell\in [u_\ell,u_\ell+\Delta_\ell^{-1}], \ell\leq \mathcal L)\right]\\
    &\hspace{2cm} 
    \leq 2\sum_{\mathbf u \in \mathcal I}\E\left[\Big|\zeta\mathcal M -1\Big|^2\prod_{\ell\leq \mathcal L} |\mathcal D_+(Y_\ell-u_\ell)|^2\right]
  \end{align*}
  where the inequality follows by an application of Lemma \ref{lem: approx}. 
  In view of the choice of $\mathfrak s$ in \eqref{eqn: choice of s}, the product is a well-factorable polynomial and of length admissible to apply Lemma \ref{lem:mollification}, cf.~Equation \eqref{eqn: length D}.
  The above is then
  \[
  \ll\Big(\frac{e^{-2\delta}}{\delta}+e^{-\mu_\mathcal L/10}\Big)\sum_{\mathbf u \in \mathcal I}\E\left[\prod_{\ell\leq \mathcal L} |\mathcal D_+(Y_\ell-u_\ell)|^2\right].
  \]
  We can also replace $Y_\ell$ by $\mathcal Y_\ell$ at a cost of a small multiplication constant $1+\OO(T^{-99})$.
  It remains to re-express the polynomial in terms of the $\mathcal Y_\ell$'s as in the proof of Proposition \ref{prop: S to G}. This proves the claim.
\end{proof}

\begin{proof}[Proof of Proposition \ref{prop: zeta to M}]
  By introducing the event that $\zeta \mathcal M$ is close to $1$,  we have
  \begin{align*}
    \PP(\log|\zeta|>V)&\geq \PP\Big(\{\log |\mathcal M^{-1}|>V-\log|\zeta \mathcal M|\}\cap\{|\zeta\mathcal M-1|<\e\Big)\\
    &\geq  \PP\Big(\{\log |\mathcal M^{-1}|>V-\log(1-\e)\}\cap\{|\zeta\mathcal M-1|<\e\}\Big)
  \end{align*}
  If we also intersect with $G(V)$, we get the lower bound
  \[
  \geq \PP\Big(\{\log |\mathcal M^{-1}|>V-\log(1-\e)\}\cap G(V)\Big)-  \PP\Big(\{|\zeta\mathcal M-1|>\e\}\cap G(V)\Big)
  \]
  The second term is bounded by Markov's inequality
  \[
  \PP\Big(\{|\zeta\mathcal M-1|>\e\}\cap G(V)\Big)\leq \frac{1}{\e^2}\E\Big[|\zeta\mathcal M-1|^2 \1(G(V))\Big]
  \]
  The claimed bound then follows from Lemma \ref{lem: mollif adapted}. 
\end{proof}

\subsection{Proof of Proposition \ref{prop: M to S}}

The first step consists in applying Lemma \ref{lem: molli approx}.
For this, it is necessary to consider the complex partial sums
\begin{equation}
  \label{eqn: S tilde}
  \wS_\ell=\sum_{T_0<p \leq T_\ell} \frac{ p^{-\ii \tau}}{p^{\sigma}}+\frac{ p^{-2\ii\tau}}{2p^{2\sigma}}, \quad \ell\geq 0.
\end{equation}
We consider the decreasing events
\[
A_\ell=A_{\ell-1}\cap\left\{|\wS_{\ell}-\wS_{\ell-1}|\leq 10\alpha(t_\ell-t_{\ell-1})\right\}, \ 1\leq \ell\leq \mathcal L,
\]
with $A_0$ being the whole sample space. On $A_{\mathcal L}$, Lemma \ref{lem: molli approx} applied with $A=10(\alpha\vee 1)$ implies
\[
\prod_{1\leq \ell\leq \mathcal L}
|\mathcal M_\ell^{-1}|\geq \exp\Big(S_{\mathcal L}+\OO^\star(5T_0^{-1/2})-\sum_{\ell=1}^{\mathcal L} e^{-10(\alpha\vee 1)(t_\ell-t_{\ell-1})}\Big)\geq \exp(S_{\mathcal L}-(\alpha\vee 1)^{-1}),
\]
where we use the notation of Lemma~\ref{lem: molli approx}, writing $\OO^\star$ for an error term with implicit constant smaller than $1$. These observations together means that
\begin{align*}
  \PP\left(\{ \log |\mathcal M^{-1}|>V\} \cap G(V)\right)
  &\geq  \PP\left( S_{\mathcal L}+\log |\mathcal M_0|^{-1}>V+(\alpha\vee 1)^{-1} \}\cap G(V)\right)-\PP\left( G_\mathcal L\cap A_\mathcal L^c\right)\\
  &\geq  \PP\left( G(V+(\alpha\vee 1)^{-1})\right)-\PP\left( G_\mathcal L\cap A_\mathcal L^c\right)\\
  &\gg  \PP\left( \mathcal Z_0+\mathcal Z_\mathcal L>V\right)-\PP\left( G_\mathcal L\cap A_\mathcal L^c\right).
\end{align*}
The second inequality is by definition of $G(V)$ and the third is by Propositions \ref{prop: G} and \ref{prop: S to G}.
It remains to bound the second probability. 
It is smaller than 
\[
\sum_{\ell=1}^{\mathcal L}\PP\Big(G_{\ell-1}\cap\{|\wS_{\ell}-\wS_{{\ell-1}}|>10\alpha (t_\ell-t_{\ell-1})\}\Big),
\]
with the convention that $G_0$ is the whole sample space $[T,2T]$.
Markov's inequality for fixed $q\in \N$ can be applied to each summand to yield
\[
\leq
\frac{\E[|\wS_{\ell}-\wS_{{\ell-1}}|^{2q}\1(G_{\ell-1})]}{|10\alpha (t_\ell-t_{\ell-1})|^{2q}}
\ll
\frac{\E[|\wS_{\ell}-\wS_{{\ell-1}}|^{2q}]}{|10\alpha (t_\ell-t_{\ell-1})|^{2q}}\cdot \PP\big(\mathcal Z_{\ell-1}\geq L_{\ell-1}-(\alpha\vee 1)^{-1}, \mathcal Z_0>\alpha t_0\big),
\]
for $\ell >1$ by Corollary \ref{cor: decoupling gaussian}. 
The factor $(\alpha\vee 1)^{-1}$ is irrelevant by \eqref{eqn: heuristic}. For $\ell=1$, we simply use the bound $ \frac{\E[|\wS_{\ell}-\wS_{{\ell-1}}|^{2q}]}{|10\alpha (t_\ell-t_{\ell-1})|^{2q}}$.
The $2q$-moment can be estimated using Lemma \ref{lem: sound moments}. Picking $q=\lfloor \alpha^2 (t_\ell-t_{\ell-1})\rfloor$, 
this yields the upper bound
\begin{align*}
  &\ll \alpha \sqrt{t_\ell-t_{\ell-1}} \exp\Big(\frac{-100\alpha^2(t_\ell-t_{\ell-1})^2}{t_\ell-t_{\ell-1}}\Big) \PP(\mathcal Z_{\ell-1}\geq L_{\ell-1},\mathcal Z_0>\alpha t_0)\\
  &\ll e^{-99\alpha^2(t_\ell-t_{\ell-1})}  \frac{e^{-2\alpha^2\delta}}{\alpha\sqrt{t}}\exp(-\alpha^2(t_{\ell-1}-t_0)+2\alpha \mathcal C\log _{\ell-1} t)\PP(\mathcal Z_0>\alpha t_0)\\
  &\ll \frac{ e^{-2\alpha^2\delta} e^{-\alpha^2t} }{\alpha\sqrt{t}}e^{-10\alpha^2 \mathfrak s\log_{\ell-1}t}\cdot e^{\alpha^2 t_0}\PP(\mathcal Z_0>\alpha t_0).
\end{align*}
The second inequality follows by a Gaussian estimate and the variance estimate \ref{eqn: variance}. 
The factor $e^{\alpha^2 t_0}\PP(\mathcal Z_0>\alpha t_0)$ is $\ll a_\alpha e^{\gamma \alpha^2}$ by \eqref{eqn: P Z_0}.
Summing the above bound on $1\leq \ell\leq \mathcal L$ yields
\[
\PP\left( G_\mathcal L\cap A_\mathcal L^c\right)\ll  \frac{C_\alpha(\delta)}{\alpha\sqrt{t}}e^{-\alpha^2 t} e^{-10\alpha^2 \mathfrak s\log_{\mathcal L}t}
\]
By the choice of $\mathfrak s$, we thus have
\[
\PP\left(\{ \log |\mathcal M^{-1}|>V\} \cap G(V)\right)\gg  \PP\left( \mathcal Z_0+\mathcal Z_\mathcal L>V\right)-\PP\left( G_\mathcal L\cap A_\mathcal L^c\right)\gg 
C_\alpha(\delta)\frac{e^{-\alpha^2 t} }{\alpha\sqrt{t}},
\]
as claimed

\section{Proof of Corollary \ref{cor: moments}}
\label{sect: cor}
We first establish that, for $k\geq 0$, $\sigma =1/2 + \delta e^{-t_\mathcal{L}}$, and $\delta$ satisfying the conditions of Theorem~\ref{thm: selberg LB},
\begin{equation}\label{eq: off-axis mom}
  \E[|\zeta(\sigma +\ii \tau)|^{2k}]\geq c_k(\log T)^{k^2},
\end{equation}
with the coefficient proportional to 
\begin{equation}\label{eq: coeff_restated}
  a_k e^{\gamma k^2}(k+1)^{-38k^2}. 
\end{equation}
Firstly, we have 
\begin{align*}
  \E[|\zeta(\sigma +\ii \tau)|^{2k}]
s  &=\left.x^{2k}\PP(|\zeta(\sigma+\ii \tau)|>x)\right|_0^\infty +2k\int_\mathbb{R}e^{2ky}\PP(|\zeta(\sigma+\ii\tau)|>e^y)dy\\
  &\geq 2k \int_{\alpha^-}^{\alpha^+} e^{2k\alpha t}\PP(\log|\zeta(\sigma + \ii \tau)|\geq \alpha t)t\;\rd \alpha\\
  &\gg 2k\int_{\alpha^-}^{\alpha^+} \frac{C_\alpha(\delta)}{\alpha}e^{2k \alpha t}e^{-\alpha^2 t}\sqrt{t}\;\rd \alpha,
\end{align*}
by  localising to a window around $\alpha t$ and employing Theorem~\ref{thm: selberg LB}.  The values $\alpha^-, \alpha^+$ can be taken to be
\[\alpha^-=k-\frac{1}{\sqrt{t}},\qquad \alpha^+=k.\]
For $\alpha> 0$, $C_\alpha(\delta)$ is decreasing and therefore the above is greater than
\[
\gg C_k(\delta)e^{k^2t}\int_{\alpha^-}^{\alpha^+}e^{-(\alpha-k)^2 t}\sqrt{t}\;\rd \alpha \gg C_k(\delta) e^{k^2t}.\]
As noted in \eqref{eqn: heuristic2} the coefficient grows as \eqref{eq: coeff_restated} using the optimal $\delta\sim \log(\alpha+1)$, completing the proof of \eqref{eq: off-axis mom}. 

The corollary follows by demonstrating that the moments on-axis are an upper bound for the moments off-axis.  To do so, we employ a result from~\cite{argouirad21}, a generalization of a convexity result of Gabriel~\cite{gab27}. 

\begin{prop}[Corollary 3.4 of~\cite{argouirad21}]\label{prop: convexity}
  Write $z=v+it$.  Suppose that $F$ is a complex valued function, analytic in the half-plane $v\geq 1/2$, such that $|F(z)|\rightarrow 0$ uniformly in $v> 1/2$ as $t\rightarrow\infty$ and
  \[\int_{\mathbb{R}}|F(v+\ii t)|^k\rd t\rightarrow 0\]
  as $v\rightarrow\infty$, for $k>0$. Then for any $v\geq 1/2$ and $k>0$
  \[\int_{\mathbb{R}}|F(v+\ii t)|^k\rd t\leq \int_{\mathbb{R}}|F(1/2+\ii t)|^k\rd t.\]
\end{prop}
The purpose of the following lemma, an adaptation of Lemma 3.5 of~\cite{argouirad21}, is to construct an analytic approximation for the indicator function of the rectangle
\[\{\nu+\ii t: \nu\in[\tfrac{1}{2},\sigma], t\in[T,2T]\}.\]
\begin{lem}[Adapted from~\cite{argouirad21}]\label{lem: rectangle indicator}
  Let $b_1\in(0,1)$ and $\Delta, T, A, b_2>0$.  There exists an entire function $\Phi_{\Delta,T}(z)\equiv \Phi(z)$ such that, for $z=\nu+\ii t$ with $\nu\geq \frac{1}{2}$ and $t\in\mathbb{R}$,
  \begin{enumerate}[(i)]
  \item For $|t-\frac{3}{2}T|\geq \frac{3}{4}(1+b_2)T$, uniformly in $\nu\geq\frac{1}{2}$, $\Phi(z)\ll_A  b_2^{-A}\Delta^{1-A}$.
  \item For $|t-\frac{3}{2}T|\leq \frac{3}{4}(1-b_1)T$, $|\Phi(z)|=1+\OO_{b_1,A}(\Delta^{-A})+\OO((\nu-\tfrac{1}{2})\frac{\Delta^2}{T})$.
  \item For $|t-\frac{3}{2}T|\leq \frac{3}{4}(1+b_2)T$, $|\Phi(z)|\ll 1+(\nu-\frac{1}{2})\frac{\Delta^2}{T}$.
  \item $\Phi(z)\rightarrow 0$ uniformly in $t$ as $\nu\rightarrow\infty$.   
  \end{enumerate}
\end{lem}
In order to apply Proposition~\ref{prop: convexity}, we use a Dirichlet polynomial approximation of $\zeta(\sigma+\ii t)$.  Define for $\sigma\geq\frac{1}{2}$, $t\asymp T$
\[D(\sigma + \ii t) =
\begin{cases}
  \sum_{n\leq T^2}n^{-\sigma-\ii t}w\left(\log\frac{e^{100} n}{T^2}\right),&\frac{1}{2}\leq \sigma\leq 2\\
  \sum_{n\leq T}n^{-\sigma-\ii t},&\sigma>\frac{1}{2},
\end{cases}
\]
where the smoothing $w(x)$ is
\[w(x) =\frac{1}{2\pi \ii}\int_{2-\ii\infty}^{2+\ii\infty}e^{-xz}\left(\frac{e^z-1}{z}\right)^{100}\frac{\rd z}{z}.\]
By Lemma 2.8 of~\cite{argouirad21}, we then have
\begin{equation}\label{eq: zeta_to_dp}
  \zeta(\sigma+\ii t)=D(\sigma+\ii t)+\OO\left(\tfrac{1}{T}\right).
\end{equation}
Set $F(z)=D(z)\Phi(z)$.  By properties (i) and (iv) of Lemma~\ref{lem: rectangle indicator}, the requirements of Proposition~\ref{prop: convexity} are satisfied. Thus for any $k>0$ we have the `smoothed' version of the required bound:
\begin{equation}\label{eq: off to on axis}
  \int_{\mathbb{R}}|D(\sigma+\ii t)\Phi_{\Delta, T}(\sigma+\ii t)|^{2k}\rd t\leq\int_{\mathbb{R}}|D(1/2+\ii t)\Phi_{\Delta, T}(1/2+\ii t)|^{2k}\rd t.
\end{equation}
The aim is now to unsmooth both sides using Lemma~\ref{lem: rectangle indicator}.

By property (ii), with $\Delta \asymp T\log T$ and $b_1=\frac{1}{3}$, and the approximation \eqref{eq: zeta_to_dp}, we get the lower bound
\begin{equation}\label{eq: off-axis to smooth}
  \int_{T}^{2T}|\zeta(\sigma+\ii t)|^{2k}\rd t\ll\int_{T}^{2T}|D(\sigma+\ii t)|^{2k}\rd t\ll\int_{\mathbb{R}}|D(\sigma+\ii t)\Phi_{\Delta, T}(\sigma+\ii t)|^{2k}\rd t.
\end{equation}

The upper bound follows similarly, using properties (i), (iii), and (iv) of Lemma~\ref{lem: rectangle indicator},
\begin{align*}
  \int_{\mathbb{R}}|D(1/2+\ii t)\Phi_{\Delta,T}(1/2+\ii t)|^{2k} \rd t &\ll_{A} \Delta^{2k(1-A)}b_2^{-2k A}\int_{|t-\frac{3}{2}T|>\frac{3}{4}T(1+b_2)} |D(1/2+\ii t)|^{2k}\rd t \\
  &\qquad+ \int_{|t-\frac{3}{2}T|\leq\frac{3}{4}T(1+b_2)}|D(1/2+\ii t)|^{2k} \rd t\\
  &\ll \int_{|t-\frac{3}{2}T|\leq T}|D(1/2+\ii t)|^{2k} \rd t\\
  &\ll \int_{T}^{2T}|\zeta(1/2+\ii t)|^{2k} \rd t
\end{align*}
choosing $A$ large enough and $b_2=1/3$. Combining with \eqref{eq: off-axis mom}, \eqref{eq: off to on axis}, and \eqref{eq: off-axis to smooth}, we arrive at the claimed statement.

\appendix

\section{A Mollification Lemma}
This section contains a mollification lemma that was first developed by one of the authors together with Bourgade and Radziwi{\l\l} in conjunction with the work \cite{argbourad23}. 
The lemma was in the end no longer needed in the proof and was therefore unpublished. 
We include it here  as we believe that it is of independent interest.
The content of the lemma is to show that $\zeta$ can be well approximated by $\mathcal M^{-1}$ for some appropriate mollifier $\mathcal M$.
The error of the approximation can be made smaller if this approximation is made on an event represented by a Dirichlet polynomial $\mathcal Q$.
In that regard, the lemma can be seen as a substantial extension of Lemma 4.2 of \cite{abbrs19}, with additional technical difficulties coming from the presence of the arbitrary Dirichlet polynomial $\mathcal Q$. 
To state the lemma, we first recall the following definition from \cite{argbourad20}. 

We say a Dirichlet polynomial
$\mathcal{Q}$ is {\it degree-$\mathcal L$ well-factorable} if it can be written as
\[
\prod_{0 \leq \ell\leq \mathcal L} \mathcal{Q}_{\ell}(s)
\]
where
\begin{align*}
  \mathcal{Q}_\ell(s) = \sum_{\substack{p | m \Rightarrow p \in (T_{\ell - 1}, T_\ell] \\ \Omega_{\ell}(k) \leq \nu_{\ell}}} \frac{\gamma_\ell(k)}{k^s},\quad \nu_{\ell}=\frac{1}{10}\mu_\ell
\end{align*}
and the coefficients $\gamma_\ell(k)$ are arbitrary.  Here, $\Omega_\ell(k)$ is the number of prime factors of $k$ in the interval $ (T_{\ell-1},T_\ell]$, and $\mu_\ell$ defines the length of the mollifier in \eqref{eqn: M}.
Note that at $\ell=0$, there are no restrictions on the number of prime factors in $\mathcal M_0$ given in \eqref{eqn: M0}, so $\mu_0=\infty$ effectively.

A degree-$\ell$ well-factorable Dirichlet polynomial $\mathcal{Q}(s)=\sum_k\frac{\gamma(k)}{k^s}$ is short for our purpose since the length is
\begin{align}
  \leq \prod_{\ell \leq \mathcal L}T_\ell^{\nu_{\ell}}\leq \exp\Big(e^t\sum_{\ell\leq \mathcal L}\frac{10(\alpha\vee 1)\mathfrak{s}\log_{\ell-1}t)}{(\log_{\ell-1} t)^\mathfrak s}\Big) \leq \exp \Big ( e^t 20(\alpha\vee 1)\mathfrak s(\log_{\mathcal L-1} t)^{1-\mathfrak{s}} \Big )\leq T^{1/100}\label{eqn: shortQ},
\end{align}
for any $\ell\leq \mathcal L$ by the choice of $\mathfrak s$ in \eqref{eqn: s}. 

\begin{lem}\label{lem:mollification}
  Consider the mollifiers $\mathcal M_\ell$, $\ell\leq \mathcal L$ (see \eqref{eqn: M} and \eqref{eqn: M total}) and write
  \[
  \sigma_\mathcal L=\frac{1}{2}+\frac{\delta}{e^{t_\mathcal L}}
  \]
  for some fixed $\delta>0$. 
  For any degree-$\mathcal L$ well-factorable $\mathcal{Q}$, we have for $T$ large enough
  \[
  \E\left[\Big|\zeta\mathcal{M} -1\Big|^2 |\mathcal Q|^2(\sigma_\mathcal L+\ii\tau)\right]\leq \Big(\frac{e^{-2\delta}}{\delta}+e^{-\mu_\mathcal L/10}\Big)\cdot \E\Big[|\mathcal Q|^2(\sigma_\mathcal L+\ii\tau)\Big].
  \]
\end{lem}

\begin{proof}
  To lighten the notation, we write throughout the proof $\sigma=\sigma_\mathcal L$ and recall that $\mathcal M=\prod_{0\leq \ell \leq \mathcal L}\mathcal M_\ell$.   We also sometimes drop the dependence on $\sigma_\mathcal L+\ii \tau$ in the notation. 
  The proof is by evaluating the expectations for each term in the expansion
  \begin{equation}
    \label{eqn: expansion M}
    |\zeta\mathcal M-1|^2=|\zeta\mathcal M|^2 - \zeta\mathcal M-\overline {\zeta\mathcal M}+1.
  \end{equation}

  {\noindent\bf Estimating $\E[\zeta \mathcal M|\mathcal Q|^2]$.}\\
  We prove
  \begin{equation}\label{eqn:offd}
    \E[\zeta \mathcal M|\mathcal Q|^2]=\E[|\mathcal Q|^2]\cdot(1+\OO(T^{-1/2})).
  \end{equation}
  The same obviously holds for $\E[\overline{\zeta \mathcal M}|\mathcal Q|^2]$.
  We use the approximation (see for example Theorem 1.8 in \cite{ivic85})
  \begin{equation}
    \label{eqn:approx}
    \zeta(\sigma+\ii \tau)=\sum_{n\leq T}\frac{1}{n^{\sigma+\ii \tau}}+\OO(T^{-1/2}).
  \end{equation}
  We define $a_\ell(n)=1$ if $\Omega_\ell(n)\leq \mu_{\ell}$ and all prime factors of $n$ are in $(T_{\ell-1},T_\ell]$, and $0$ otherwise.  
  For $n\in \N$, we also write $n_\ell=\prod_{p\mid n, p\in (T_{\ell-1},T_\ell]} p^{v_p(n)}$, where $v_p(n)$ stands for the $p$-valuation of $n$. 
  We then get a multiplicative function $a(n)=\prod_{\ell\leq \mathcal L} a_{\ell}(n_\ell)$. 
  As $\mu$ is also multiplicative,  we can now write the mollifier in a compact form
  \[
  \mathcal M(s)=\sum_{m}\frac{\mu(m)a(m)}{m^s}.
  \]
  By re-defining the coefficient $\gamma$ in the same way to include the restriction on prime factors, we can also write
  $
  \mathcal Q(s)=\sum_{k}\frac{\gamma(k)}{k^s},
  $
  so that
  \begin{equation}
    \label{eqn: MQ factors}
    \mathcal M |\mathcal Q|^2(s)=\sum_{m,k_1k_2} \frac{\mu(m)a(m)\gamma(k_1)\overline{\gamma}(k_2)}{(mk_1k_2)^s}.
  \end{equation}
  The contribution from the error term in \eqref{eqn:approx} is
  \begin{equation}\label{eqn:error}
    \ll T^{-1/2} \sum_{m,k_1,k_2} \frac{|\gamma(k_1)\gamma(k_2)|}{(mk_1k_2)^\sigma}\ll T^{-\frac{1}{2}+\frac{1}{100}}\left(\sum_{k}\frac{|\gamma(k)|}{k^\sigma}\right)^2
    \ll T^{-\frac{1}{2}+\frac{2}{100}}\sum_{k}\frac{|\gamma(k)|^2}{k^{2\sigma}},
  \end{equation}
  where we have used the bound \eqref{eqn: shortM} on $m$ in the second inequality, and the Cauchy-Schwarz inequality in the third.
  We now observe that $\E[|\mathcal Q|^2]$ is close to $\sum_{k}\frac{|\gamma(k)|^2}{k^{2\sigma}}$ since
  \begin{align*}
    \E[|\mathcal{Q}|^2]-\sum_{k}\frac{|\gamma(k)|^2}{k^{2\sigma}}=\sum_{k_1\neq k_2}\frac{\gamma(k_1)\overline\gamma(k_2)}{(k_1k_2)^\sigma} \E\left[\left(\frac{k_2}{k_1}\right)^{\ii \tau}\right]
    &\ll
    \sum_{k_1,k_2}\frac{|\gamma(k_1)\gamma(k_2)|}{(k_1k_2)^\sigma} \frac{k_2}{T}\\
    &\ll
    T^{-1+\frac{1}{100}} \sum_{k}\frac{|\gamma(k)|^2}{k^{2\sigma}}.
  \end{align*}
  We used  $\E[(\frac{k_2}{k_1})^{\ii \tau}]\ll 1/(T|\log(k_1/k_2)|)\ll \frac{k_2}{T}$,  \eqref{eqn: shortQ} to bound $k_2$, and the Cauchy-Schwarz inequality.
  This shows that 
  \begin{equation}\label{eqn:Q}
    \sum_{k}\frac{|\gamma(k)|^2}{k^{2\sigma}}=(1+\OO(T^{-1/2}))\E[|\mathcal{Q}|^2].
  \end{equation}
  The above equation together with \eqref{eqn:error} then yields
  \begin{equation}
    \E\Big[\zeta\mathcal M |\mathcal Q|^2(\sigma+\ii\tau)\Big]
    =\sum_{n\leq T,m,k_1k_2} \frac{\mu(m)a(m)\gamma(k_1)\overline{\gamma}(k_2)}{(nmk_1k_2)^{\sigma}}\E\left[\Big(\frac{k_2}{nmk_1}\Big)^{\ii \tau}\right] \label{eqn:hol}
    +\OO(T^{-1/2})\E[|\mathcal{Q}|^2].
  \end{equation}
  We first bound the terms $nm\neq k_2/k_1$, based on 
  $\E[(\frac{k_2}{nmk_1})^{\ii \tau}]\ll 1/(T|\log\frac{nmk_1}{k_2}|)\ll\frac{k_2}{T}$,  and \eqref{eqn: shortQ} and  \eqref{eqn: shortM}  as before:
  \begin{equation}
    \label{eqn:offdiag}
    \begin{aligned}
      \sum_{nm\neq k_2/k_1} \frac{\mu(m)a(m)\gamma(k_1)\overline{\gamma}(k_2)}{(mn)^{\sigma}} \E\left[\left(\frac{k_2}{k_1nm}\right)^{\ii \tau}\right]
      &\ll
      T^{\frac{1}{100}}\sum_{n\leq T,k_1,k_2} \frac{\gamma(k_1)\overline{\gamma}(k_2)}{(mn)^{\sigma}}\frac{k_2}{T}\\
      &\ll T^{-1+\frac{2}{100}} \sum_{k}\frac{|\gamma(k)|^2}{k^{2\sigma}}\\
      &\ll T^{-1+\frac{2}{100}}\cdot \E[|\mathcal{Q}|^2],
    \end{aligned}
  \end{equation}
  where \eqref{eqn:Q} was used for the last bound.
  
  It remains to estimate the contribution of the terms $nm=k_2/k_1$. Note that the restriction $n\leq T$ is no longer needed here since it is implied by $n=k_2/(k_1 m)$. 
  Moreover,  $m=k_2/(k_1n)$ and $\Omega_\ell(k_2)\leq \mu_\ell$ imposes $a(m)=1$. With these simple observations, we are left to estimate,
  for fixed $k_1$ dividing $k_2$, the sum
  $
  \sum_{nm=k_2/k_1}\mu(m)=\mathbf {1}_{k_2=k_1}.
  $
  With Equations~\eqref{eqn:Q},~\eqref{eqn:hol} and~\eqref{eqn:offdiag}, this  gives \eqref{eqn:offd}.
  
  \bigskip
  
  {\noindent\bf Estimating $\E[|\zeta \mathcal M\mathcal Q|^2]$.}\\
  We will now prove the more delicate asymptotics
  \begin{equation}\label{eqn:dia}
    \E[|\zeta\mathcal M\mathcal Q|^2]=\E[|\mathcal Q|^2]\cdot \left(1+ \OO\Big(\frac{e^{-2\delta}}{\delta}+e^{-\mu_\mathcal L}\Big)\right),
  \end{equation}
  relying on the expansion
  \begin{equation}\label{eqn:dia2}
    \E[|\zeta\mathcal M\mathcal Q|^2]=\sum_{k_1,k_2,m_1,m_2}\frac{\mu(m_1)\mu(m_2)a(m_1)a(m_2)\gamma(k_1)\overline{\gamma(k_2)}}{(m_1m_2k_1k_2)^\sigma}\E\left[|\zeta|^2\left(\frac{k_1 m_1}{k_2m_2}\right)^{\ii\tau}\right],
  \end{equation}
  and on the estimate (see \cite[Lemma 4]{radsou17})
  \begin{align*}
    \int_T^{2T}\Big(\frac{n_1}{n_2}\Big)^{\ii t}\left|\zeta\left(\sigma+\ii t\right)\right|^2\rd t=\int_T^{2T}&\Big(
    \zeta(2\sigma)\Big(\frac{(n_1,n_2)^2}{n_1 n_2}\Big)^\sigma
    +
    \Big(\frac{t}{2\pi}\Big)^{1-2\sigma}\zeta(2-2\sigma)\Big(\frac{(n_1,n_2)^2}{n_1n_2}\Big)^{1-\sigma}
    \Big)\rd t\\
    &+\OO(T^{1-\sigma+\e}\min(n_1,n_2)).\numberthis \label{eqn:esti}
  \end{align*}
  As $k_1,m_1,k_2,m_2\leq T^{1/100}$, and $|a(m_1)|,|a(m_2)|\leq 1$,  the contribution in \eqref{eqn:dia2} from the error term in \eqref{eqn:esti}  is
  \[
  \ll T^{-\frac{1}{4}}\sum_{k_1,k_2}\frac{|\gamma(k_1)\gamma(k_2)|}{(k_1k_2)^{\sigma}}\ll T^{-\frac{1}{10}}\sum_k\frac{|\gamma(k)|^2}{k^{2\sigma}}\leq T^{-\frac{1}{100}}\cdot  \E[|\mathcal{Q}|^2],
  \]
  where we have used \eqref{eqn:Q} in the last inequality.\\
  
  \noindent {\it Contribution from the right of $\re \ s=1$}\\
  We now consider the contribution to \eqref{eqn:dia2}  from the term $\zeta(2\sigma)\Big(\frac{(h,k)^2}{hk}\Big)^\sigma$ in \eqref{eqn:esti}:
  \begin{equation}
    \label{eqn:S1}
    \zeta(2\sigma)\sum_{k_1,k_2}\frac{\gamma(k_1)\overline{\gamma(k_2)}}{(k_1k_2)^{2\sigma}}\sum_{m_1,m_2}\frac{\mu(m_1)\mu(m_2)a(m_1)a(m_2)}{(m_1m_2)^{2\sigma}}(k_1m_1,k_2m_2)^{2\sigma}.
  \end{equation}
  We first notice that we can estimate the sums for each interval of primes $(T_{\ell-1},T_\ell]$ and take the product on $\ell$ at the very end, by the definition of the $a$'s and the fact that $\mathcal Q$ is well-factorable. 
    
  With this in mind, let's first estimate the sum without the restriction on  $\Omega_\ell$. 
  Fix $k_1$ and $k_2$. 
  Assume first that $k_1\neq k_2$. Without loss of generality, we can assume that $(k_1,k_2)=1$, otherwise we can factor it out. We write $\mathcal P_1$ for the set of prime factors of $k_1$ and similarly for $k_2$. 
  The sum over $m_1$ and $m_2$ without the restriction can then be expressed as an Euler product as follows
  \begin{align*}
    &\sum_{\substack{m_1,m_2\\p| m_i \Rightarrow p \in (T_{\ell-1},T_\ell],i=1,2}}\frac{\mu(m_1)\mu(m_2)}{(m_1m_2)^{2\sigma}}(k_1m_1,k_2m_2)^{2\sigma}\\
      &=
      \prod_{p\in(\mathcal{P}_1\cup\mathcal{P}_2)^c}(1-p^{-2\sigma}-p^{-2\sigma}+p^{-2\sigma})
      \prod_{p\in\mathcal{P}_1}(1-1+p^{-2\sigma}-p^{-2\sigma})
      \prod_{p\in\mathcal{P}_2}(1-1+p^{-2\sigma}-p^{-2\sigma}).\numberthis\label{eqn:keyvanishing}
  \end{align*}
  This is obviously $0$. With this crucial observation, we have that 
  \begin{align*}
    \sum_{p | k_i \Rightarrow p \in (T_{\ell-1},T_\ell],i=1,2}\frac{\gamma(k_1)\overline{\gamma(k_2)}}{(k_1k_2)^{2\sigma}}&\sum_{p | m_i \Rightarrow p \in (T_{\ell-1},T_\ell],i=1,2}\frac{\mu(m_1)\mu(m_2)}{(m_1m_2)^{2\sigma}}(k_1m_1,k_2m_2)^{2\sigma}\\
        &=\sum_{p | k \Rightarrow p \in (T_{\ell-1},T_\ell]}\frac{|\gamma(k)|^2}{k^{2\sigma}}\sum_{p | m_i \Rightarrow p \in (T_{\ell-1},T_\ell],i=1,2}\frac{\mu(m_1)\mu(m_2)}{(m_1m_2)^{2\sigma}}(m_1,m_2)^{2\sigma}\\
            &=(1+\OO(T^{-1/2}))\E[|\mathcal{Q}_\ell|^2]\prod_{p\in (T_{\ell-1},T_\ell]}\left(1-\frac{1}{p^{2\sigma}}\right),\numberthis \label{eqn: dominant Qell}
  \end{align*}
  where we used \eqref{eqn:Q}. 
  
  We now bound the contribution from integers $m_1$ with $\Omega_\ell(m_1)>\mu_\ell$ for a fixed $\ell\geq 1$.
  We use Rankin's trick, i.e., $\1(\Omega_\ell(m)>\mu_\ell)\leq e^{-\mu_\ell}e^{\Omega_\ell(m)}$.
  We get that the sum is bounded by
  \begin{equation}
    \label{eqn: prod 1}
    e^{-\mu_\ell}\sum_{\substack{m_1,m_2,k_1,k_2\\p | k_i,m_i \Rightarrow p \in (T_{\ell-1},T_\ell],i=1,2}}
      \frac{\gamma(k_{1})\overline{\gamma(k_{2})}}{(k_{1}k_{2})^{2\sigma}}\cdot \frac{|\mu(m_{1})\mu(m_{2})|e^{\Omega_\ell(m_1)}}{(m_{1}m_{2})^{2\sigma}}(k_{1}m_{1},k_{2}m_{2})^{2\sigma}.
  \end{equation}
  The sum over $m_1,m_2$ is computed similarly to \eqref{eqn:keyvanishing}. It is 
  \begin{align*}
    &=(k_1,k_2)^{2\sigma}\prod_{\mathcal{P}_1}\left(2+\frac{2e}{p^{2\sigma}}\right)
    \prod_{\mathcal{P}_2}\left(1+e+\frac{1+e}{p^{2\sigma}}\right)
    \prod_{(\mathcal{P}_1\cup\mathcal{P}_2)^{c}}\left(1+\frac{2e+1}{p^{2\sigma}}\right)\\
    &\leq  (k_1,k_2)^{2\sigma} (1+e)^{\omega_\ell(k_1)+\omega_\ell(k_2)} \prod_{p\in (T_{\ell-1},T_\ell]} \Big(1+\frac{2e+1}{p^{2\sigma}}\Big),
  \end{align*}
  where $\omega_\ell(m)$ stands for the number of distinct prime factors of $m$ in $(T_{\ell-1},T_\ell]$. 
  Using the fact that $|\gamma(k_{1})\gamma(k_{2})|\leq (|\gamma(k_{1})|^2+|\gamma(k_{2})|^2)/2$, the sum in \eqref{eqn: prod 1} is bounded by
  \begin{equation}
    \label{eqn:simple}
    \prod_{p\in (T_{\ell-1},T_\ell]} \Big(1+\frac{2e+1}{p^{2\sigma}}\Big)\sum_{\substack{k_{1}\\p | k_1 \Rightarrow p \in(T_{\ell-1},T_\ell]}}
        \frac{|\gamma(k_{1})|^2}{k_{1}^{2\sigma}}(1+e)^{\omega_\ell(k_1)} \sum_{\substack{k_{2}\\p | k_2 \Rightarrow p \in (T_{\ell-1},T_\ell]}}\frac{(k_1,k_2)^{2\sigma} }{k_{2}^{2\sigma}} (1+e)^{\omega_\ell(k_2)} .
  \end{equation}
  For $k_1$ fixed, we can express the sum over $k_2$ as an Euler product
  \begin{align*}
    &=\prod_{p| k_1}\Big(1+(1+e)v_p(k_1)+(1+e)\sum_{j>v_p(k_1)}p^{-2\sigma (j-v_p(k_1))}\Big)\prod_{p\nmid k_1}\Big(1+(1+e)\sum_{j\geq 1}p^{-2\sigma j}\Big)\\
    &\leq (1+e)^{\omega_\ell(k_1)}\prod_{p\mid k_1}(1+v_p(k_1)) \prod_{p\nmid k_1}\Big(1+\frac{2e+1}{p^{2\sigma}}\Big)\\
    &\leq (1+e)^{\Omega_\ell(k_1)}e^{\Omega_\ell(k_1)} \prod_{p\in (T_{\ell-1},T_\ell]}\Big(1+\frac{2e+1}{p^{2\sigma}}\Big),
  \end{align*}
  where we use the fact that $1+x\leq e^x$ to estimate the product over $p|k_1$. Using the assumption that $\Omega_\ell(k_1)\leq \nu_\ell$, we can finally bound \eqref{eqn: prod 1} by
  \begin{equation}
    \label{eqn: subdominant Qell}
    \leq (1+\OO(T^{-1/2}))(1+e)^{3\nu_\ell-\mu_\ell} \prod_{p\in (T_{\ell-1},T_\ell]}\Big(1+\frac{2e+1}{p^{2\sigma}}\Big)^2\E[|\mathcal Q_\ell|^2],
  \end{equation}
  where we used \eqref{eqn:Q}. By the choice of $\nu_\ell$, we have that $3\nu_\ell-\mu_\ell<\mu_\ell/2$.
  
  To estimate \eqref{eqn:S1}, it remains to take the product over $\ell$ and multiply by ${\zeta(2\sigma)=\prod_{p}(1-p^{-2\sigma})^{-1}}$. 
  We now apply the identity
  \[
  \prod_{\ell\leq \mathcal L}(1-\1(\Omega_\ell(m)>\mu_\ell))=\sum_{I\subseteq \{1,\dots,\mathcal L\}}\prod_{\ell\in I}(-1)^{|I|}\1(\Omega_\ell(m)>\mu_\ell)
  \]
  to the sums in $m_1$ and $m_2$ in \eqref{eqn:S1} to get that it is equal to
  \begin{multline}
    \label{eqn: O term}
    \E[|\mathcal{Q}|^2]\Big\{(1+\OO(T^{-1/2}))\prod_{p>T_\mathcal L}\left(1-\frac{1}{p^{2\sigma}}\right)^{-1} -\\
    \OO^{\star}\Big(\sum_{I_1\cup I_2\neq \emptyset}\prod_{\ell\in I_1 \cup I_2}e^{-\mu_\ell/2}\prod_{p\in (T_{\ell-1},T_\ell]}\Big(1+\frac{2e+1}{p^{2\sigma}}\Big)^4\Big(1-\frac{1}{p^{2\sigma}}\Big)^{-1}\Big)
      \Big\}.
  \end{multline}
  Here the sums is over subsets $I_1$ and $I_2$ whose union is not empty. They represent the $\ell$'s for which $\Omega_\ell(m)>\mu_\ell$ for $m_1$ and $m_2$ respectively. As in Lemma~\ref{lem: molli approx}, $\OO^\star$ represents an error term with implicit constant less than $1$. 
  We then use the estimate \eqref{eqn: dominant Qell} when $\ell\in (I_1\cup I_2)^c$ and the estimate \eqref{eqn: subdominant Qell} when $\ell \in I_1\cup I_2$. Note that in the case where $\Omega_\ell(m)>\mu_\ell$ for both $m_1$ and $m_2$, 
  we need to square the term $\Big(1+\frac{2e+1}{p^{2\sigma}}\Big)^2$. 
  The first product in the braces is estimated by the Prime Number Theorem:
  \begin{align*}
  \sum_{p>T_\mathcal L}\frac{1}{p^{2\sigma}}\leq \int_{T_\mathcal L}^\infty \frac{1}{u^{2\sigma}\log u}\rd u +\OO(e^{-c\sqrt{T_\mathcal L}})
  &=\int_{\log T_\mathcal L}^\infty \frac{e^{-v(2\sigma-1)}}{v}\rd v +\OO(e^{-c\sqrt{T_\mathcal L}})\\
  &\leq\frac{e^{-2\delta}}{2\delta}+\OO(e^{-c\sqrt{T_\mathcal L}}),
  \end{align*}
  by the definition of $\sigma$ and $\delta$. Therefore, $\prod_{p>T_\mathcal L}\left(1-p^{-2\sigma}\right)^{-1}$ is $1+e^{-2\delta}/\delta$ for $T$ large enough.
  
  Write $\rho_\ell=\prod_{p\in (T_{\ell-1},T_\ell]}\Big(1+(2e+1)p^{-2\sigma}\Big)^4\Big(1-p^{-2\sigma}\Big)^{-1}$. We have that locally the factors are bounded by $(1+100p^{-2\sigma})$. 
  Moreover, Mertens's theorem implies $e^{-\mu_\ell/2}\rho_\ell\leq e^{-\mu_\ell/4}$ (for $\mathfrak s>100$). 
  The $\OO^\star$-term is then 
  \[
  \sum_{I_1\cup I_2\neq \emptyset}\prod_{\ell\in I_1 \cup I_2}e^{-\mu_\ell/2}\rho_\ell\leq \prod_{\ell\leq \mathcal L}(1+e^{-\mu_\ell/4})^2 -1\leq e^{-\mu_\ell/10}.
  \]
  This concludes the proof  of the contribution in \eqref{eqn:dia2} from the term $\zeta(2\sigma)\Big(\frac{(n_1,n_2)^2}{n_1n_2}\Big)^\sigma$ in \eqref{eqn:esti}.\\
  
  \noindent {\it Contribution from the left of $\re\ s=1$} \\
  We now consider the contribution from $\Big(\frac{t}{2\pi}\Big)^{1-2\sigma}\zeta(2-2\sigma)\Big(\frac{(n_1,n_2)^2}{n_1n_2}\Big)^{1-\sigma}$ in \eqref{eqn:esti}. 
  The integral factor there is
  \begin{equation}
    \label{eqn: integral}
    \frac{1}{T}\int_T^{2T}\left(\frac{t}{2\pi}\right)^{1-2\sigma}\rd t\ll T^{1-2\sigma}=\exp(-2\delta e^{t-t_\mathcal L}).
  \end{equation}
  For the other terms, we can work on each interval $(T_{\ell-1}, T_\ell]$ for $\ell$ fixed as before.
  Without the restriction on the number of prime factors, the sum over $k$'s and $m$'s is
  \begin{equation}
    \label{eqn: sum left}
    \sum_{\substack{m_1,m_2,k_1,k_2\\p | k_i,m_i \Rightarrow p \in (T_{\ell-1},T_\ell],i=1,2}}
      \frac{\gamma(k_{1})\overline{\gamma(k_{2})}}{k_1k_2}\frac{\mu(m_{1})\mu(m_{2})}{m_1m_2} (k_{1}m_{1},k_{2}m_{2})^{2-2\sigma}.
  \end{equation}
  We write $k_1=(k_1,k_2)l_1$,  $k_2=(k_1,k_2)l_2$ where $l_1$ and $l_2$ are coprime with corresponding prime divisors sets $\mathcal{P}_1$ and $\mathcal{P}_2$.
  The above sum over $m_1, m_2$ is equal to 
  \begin{align*}
    &\sum_{\substack{m_1,m_2\\p | m_1,m_2 \Rightarrow p \in (T_{\ell-1},T_\ell]}}\frac{\mu(m_1)\mu(m_2)}{m_1m_2}(k_1m_1,k_2m_2)^{2-2\sigma}\\
      &\quad=(k_1,k_2)^{2-2\sigma}\prod_{p\in(\mathcal{P}_1\cup\mathcal{P}_2)^c}(1-2p^{-1}+p^{-2\sigma})
      \prod_{p\in\mathcal{P}_1}(1-p^{-1}-p^{1-2\sigma}+p^{-2\sigma})
      \prod_{p\in\mathcal{P}_2}(1-p^{-1}-p^{1-2\sigma}+p^{-2\sigma}).\numberthis \label{eqn:keyvanishing2}
  \end{align*}
  As opposed to \eqref{eqn:keyvanishing}, the above product does not vanish when $k_1\neq k_2$.
  Therefore, estimating the sum is slightly more involved. We define 
  \[c_\ell=(2\sigma-1)\log T_\ell.\]
  On $(\mathcal{P}_1\cup\mathcal{P}_2)^{\rm c}$, we bound 
  \[1-p^{-1}-p^{-1}+p^{-2\sigma}\leq 1-p^{-1},\]
  and on $\mathcal{P}_1\cup\mathcal{P}_2$, we have when $p\leq T_\ell$,
  \[1-p^{-1}-p^{1-2\sigma}+p^{-2\sigma}=(1-p^{-1})\cdot(1-p^{1-2\sigma})\leq (1-p^{-1}) c_\ell.\]
  This shows that the sum in \eqref{eqn:keyvanishing2} is
  \begin{equation}
    \label{eqn: sum left 1}
    \leq (k_1,k_2)^{2-2\sigma} \prod_{p\in(T_{\ell-1},T_\ell]}\left(1-\frac{1}{p}\right)\cdot c_\ell^{\omega_\ell(k_1)+\omega_\ell(k_2)}.
  \end{equation}
  Moreover, bounding the $\gamma$'s as before, we get that \eqref{eqn:keyvanishing2} is
  \[
  \leq 
  \sum_{\substack{k_1,k_2\\p | k_1,k_2\Rightarrow p \in (T_{\ell-1},T_\ell]}}
    \frac{|\gamma(k_{1})|^2}{k_{1}k_2}(k_{1},k_{2})^{2-2\sigma}c_\ell^{\omega_\ell(k_{1})+\omega_\ell(k_{2})}.
  \]
  The sum over $k_2$ can be evaluated for fixed $k_1$ by re-expressing it as an Euler product. We get that
  \begin{align*}
    \sum_{\substack{k_2\\p | k_2\Rightarrow p \in (T_{\ell-1},T_\ell]}}&c_\ell^{\omega_\ell(k_{2})}
      \frac{(k_{1},k_{2})^{2-2\sigma}}{k_2}\\
      &=\prod_{p\nmid k_1}\Big(1+\sum_{j=1}^{v_p(k_1)} c_\ell p^{j(1-2\sigma)} +\sum_{j>v_p(k_1)}c_\ell p^{v_p(k_1)(2-2\sigma)}{p^{-j}}\Big)\prod_{p\mid k_1} \Big(1+\sum_{j\geq 1}\frac{c_\ell}{p^j}\Big)\\
      &\leq \prod_{p\mid k_1} \Big(1+c_\ell v_p(k_1)p^{1-2\sigma}+c_\ell p^{v_p(k_1)(1-2\sigma)}\Big)\prod_{p\nmid k_1}\Big(1+\frac{c_\ell}{p-1}\Big).
  \end{align*}
  We factor $p^{(1-2\sigma)v_p(k_1)}$ from the first product to get $k_1^{1-2\sigma}$. Combined with $k_1$, this will get the factor $k_1^{-2\sigma}$ needed to recover $\mathcal Q_\ell^2$. 
  The second product is bounded by $\exp(2c_\ell(t_\ell-t_{\ell-1}))$ by Mertens's theorem. Putting this together yields the bound
  \begin{align*}
    &k_1^{1-2\sigma}e^{2c_\ell (t_\ell-t_{\ell-1})}\prod_{p\mid k_1} \Big(p^{v_p(k_1)(2\sigma-1)}+c_\ell v_p(k_1)p^{(v_p(k_1)-1)(2\sigma-1)}+c_\ell \Big)\\
    &\leq 
    k_1^{1-2\sigma}e^{2c_\ell (t_\ell-t_{\ell-1})}(c_\ell+c_\ell \nu_\ell e^{c_\ell\nu_\ell}+e^{c_\ell\nu_\ell})^{\nu_\ell}
    \leq k_1^{1-2\sigma}e^{2c_\ell (t_\ell-t_{\ell-1})+c_\ell\nu_\ell^2}.
  \end{align*}
  where we used the definition of $c_\ell$ and the bound $\Omega_\ell(k_1)\leq \mu_\ell$ to get the inequality.
  Taking the product over $\ell$ and including the integral factor \eqref{eqn: integral}, we can now conclude that the contribution to \eqref{eqn: sum left} of sums on $m$'s and $k$'s without the restriction on the number of factors of the $m$'s is
  \[
  \zeta(2-2\sigma) \prod_{p\leq T_\mathcal L}(1-p^{-1})e^{\sum_{\ell\leq \mathcal L}2c_\ell\nu_\ell^2-2\delta e^{t-t_\mathcal L}}\sum_{k_1}\frac{|\gamma(k_1)|^2}{k_1^{2\sigma}}.
  \]
  The first two factors are of order one. The sum over $k_1$ is $\ll \E[|\mathcal Q|^2]$ by \eqref{eqn:Q}. 
  Since $c_\mathcal L=2\delta$, the exponential factor is
  \[
  \leq \exp\Big(4c_\mathcal L\nu_\mathcal L^2-2\delta e^{t-t_\mathcal L}\Big)\leq \exp\Big(8\delta (10\mathfrak s \log_\mathcal Lt)^2-2\delta e^{\mathfrak s\log_\mathcal Lt}\Big)
  \leq e^{-4\delta},
  \]
  under the condition that $\mathfrak s\log_\mathcal Lt>100$. This is the dominant contribution coming from $\re s=2-2\sigma$.
  Finally, we need to bound the contribution to the sum over $m$'s of the restrictions $\Omega_{\ell}(m)>\mu_\ell$ for each $\ell$.
  This is done exactly as before starting from \eqref{eqn: prod 1}.
\end{proof}

\section{Auxiliary Results}

The following lemma is an adaption of Lemma 23 in \cite{argbourad23} that gives upper and lower bound in approximating $\mathcal M_\ell$ by $ e^{-(S_{\ell}-S_{\ell-1})}$ with precise error.
\begin{lem} \label{lem: molli approx}
  Consider the partial sums \eqref{eqn: S tilde}. Let $\ell \geq 1$ and $A>0$.
  If $|\widetilde{S}_{\ell} - \widetilde{S}_{\ell - 1}| \leq A(t_{\ell} - t_{\ell - 1})$, 
  then the mollifier $\mathcal M_\ell$ defined in \eqref{eqn: M} is satisfies
  \begin{align*}
    |\mathcal{M}_{\ell}|= \Big(1+\OO^\star\big(\tfrac{5}{\sqrt{T_{\ell-1}}}\big)\Big) e^{-(S_{\ell}-S_{\ell-1})}+e^{-\mu_\ell+5(A+ 1)(t_\ell-t_{\ell-1})+1},
  \end{align*}
  where $\OO^\star$ stands for $\OO$ with implicit constant smaller than $1$.
\end{lem}

\begin{proof}
  We set $s=\sigma+\ii\tau$. 
  We have by expanding the logarithm
  \begin{align*}
    \prod_{p \in (T_{\ell - 1}, T_\ell]} \Big ( 1 - \frac{1}{p^{s}} \Big ) & =   e^{-(S_{\ell}- S_{{\ell - 1}}) - R_\ell},
  \end{align*}
  where
  \[
  R_\ell= \sum_{a\geq 3}\sum_{p\in (T_{\ell-1},T_\ell]} \frac{1}{a} \, \re~p^{- as}. 
  \]
  We estimate $R_\ell$ as follows
  
  \begin{align*}
    |R_\ell|&\leq  \sum_{a\geq 3}\sum_{p\in (T_{\ell-1},T_\ell]} \frac{1}{a} p^{- a/2}
      \leq \frac{1}{3}\frac{\sqrt{2}}{\sqrt{2}-1} \int_{T_{\ell-1}}^\infty y^{-3/2}\rd y
      \leq 2.5 (T_{\ell-1})^{-1/2}.
  \end{align*}
  Using the fact that $e^x\leq 1+2x$ for $0<x<1/2$ and $e^{-x}\geq 1-x$ for $x<0$, we get that
  \begin{equation}
    \label{eqn: mollif S R}
    1-(T_{\ell-1})^{-1/2}\leq e^{-R_\ell}\leq 1+5(T_{\ell-1})^{-1/2}.
  \end{equation}
  It remains to relate the product over primes to the mollifier. By definition we have
  \begin{align*}
    \prod_{p \in (T_{\ell - 1}, T_\ell]} \Big ( 1 - \frac{1}{p^{s}} \Big ) 
      & = \mathcal{M}_{\ell } + \sum_{\substack{p | n \Rightarrow p \in (T_{\ell - 1}, T_\ell] \\ \Omega_{\ell - 1}(n) > \mu_\ell}} \frac{\mu(n)}{n^{s}}.
  \end{align*}
  The last term is equal to
  \begin{equation}\label{eq:tobound}
    \sum_{j > \mu_\ell} (-1)^{j} \Big ( \sum_{T_{\ell - 1} < p_1 < \ldots < p_{j} \leq T_\ell} \frac{1}{(p_1 \ldots p_j)^s} \Big ).
  \end{equation}
  The Girard-Newton identities (see for example Equation 2.14' in \cite{mac95}) show that the inner sum is
  \[
  \sum_{T_{\ell - 1} < p_1 < \ldots < p_{j} \leq T_\ell} \frac{1}{(p_1 \ldots p_{j})^s} = (-1)^j\sum_{\substack{m_1, \ldots, m_{k}, \ldots  \geq 0\\m_1 + 2 m_2 + \ldots + k m_{k} + (k + 1) m_{k+ 1} + \ldots = j }} \prod_{k\geq 1} \frac{(-\mathcal{P}(k s))^{m_k}}{m_k!\ k^{m_k}}
  \]
  where the sum is over the partitions of $j$ and
  \[
  \mathcal{P}(s) = \sum_{T_{\ell - 1} < p \leq T_\ell} \frac{1}{p^s}.
  \]
  Using this we can bound the absolute value of \eqref{eq:tobound} by
  \begin{align*} 
    \sum_{m_1, \ldots, m_{k}, \ldots \geq 0} & \exp \Big ( - \mu_\ell + m_1 + 2 m_2 + \ldots + k m_{k} + \ldots \Big ) \prod_{k\geq 1} \frac{|\mathcal{P}(ks)|^{m_k}}{m_k! \ k^{m_k}} \\\numberthis \label{eq:to} 
    & = e^{- \mu_\ell} \prod_{k\geq 1} \Big ( \sum_{m_k \geq 0} \frac{(e^{k} |\mathcal{P}(ks)| / k)^{m_k}}{m_k!} \Big )\\
    &=\exp\Big(- \mu_\ell+\sum_{k\geq 1} e^{k} |\mathcal{P}(ks)| / k \Big).
  \end{align*}
  For $\ell=1$, the assumption implies $|\mathcal{P}(s)| \leq A (t_{\ell} - t_{\ell - 1})$, giving a contribution to the sum of 
  \begin{equation}
    \label{eqn: mollif S l=1}
    \leq e A (t_{\ell} - t_{\ell - 1}).
  \end{equation}
  For $k=2$, Mertens's theorem (see \cite{van17} for a precise bound) also implies 
  \[
  |\mathcal{P}(2s)| \leq t_{\ell} - t_{\ell - 1} + \OO^\star\Big(\tfrac{4}{\log^3 T_{\ell-1}}\Big)\leq t_{\ell} - t_{\ell - 1} + 0.001,
  \]
  by the definition of $T_0$. This gives an error of 
  \begin{equation}
    \label{eqn: mollif S k=2}
    \tfrac{e^2}{2}|\mathcal{P}(2s)| \leq 5 (t_{\ell} - t_{\ell - 1}) + 0.01.
  \end{equation}
  For $k\geq 3$, we have
  \[
  |\mathcal{P}(ks)| \leq \sum_{p\in (T_{\ell - 1}, T_\ell]}p^{-k/2}\leq 2(T_{\ell-1})^{-k/2+1}.
    \]
  This implies that 
  \begin{equation}
    \label{eqn: mollif S k=3}
    \sum_{k\geq 3} e^{k} |\mathcal{P}(ks)| / k \leq   2T_0\sum_{k\geq 3} \Big(\frac{e}{T_0^{1/2}}\Big)^{k}\leq 0.001.
  \end{equation}
  The claim follows from \eqref{eqn: mollif S R}, \eqref{eqn: mollif S l=1}, \eqref{eqn: mollif S k=2}, and \eqref{eqn: mollif S k=3}.
\end{proof}

The mean-value theorem for Dirichlet polynomials, see \cite[Corollary 3]{MonVau74}, takes the following form when expressed in terms of the random model \ref{eqn: Sl random}.
Let $(\theta_p, p \text{ prime})$ a sequence of IID random variables, uniformly distributed on $[0,2\pi]$. For an integer $n$ with prime factorization $n = p_1^{a_1} \ldots p_k^{a_k}$ with $p_1, \ldots, p_k$ all distinct, consider 
\begin{equation}
  \label{eqn: theta_n}
  \theta_n = \sum_{j = 1}^{k}  \theta_{p_j}^{a_j}. 
\end{equation}
The orthogonality relation $\mathbf{E}[e^{\ii \theta_n} \overline{e^{\ii \theta_m}}] = \mathbf{1}_{n = m}$ is straightforward from the definition. This implies that for an arbitrary sequence $a(n)$ of complex numbers, 
\[
\mathbf{E} \Big [ \Big | \sum_{n \leq N} a(n) e^{\ii \theta_n} \Big |^2 \Big ]=\sum_{n \leq N} |a(n)|^2.
\]
The mean-value theorem for Dirichlet polynomials is as follows:
\begin{lem}[Corollary 3 in \cite{MonVau74}]
  \label{lem: MV}
  Let $(a(n), n\geq 1)$ be any sequence of complex numbers. Then for $N\geq 1$ and $T\geq 1$, we have with the notation above. 
  \[
  \mathbf{E} \Big [ \Big | \sum_{n \leq N} a(n) n^{\ii \tau} \Big |^2 \Big ]    = \Big ( 1 + \OO \Big ( \frac{N}{T} \Big ) \Big ) \mathbf{E} \Big [ \Big | \sum_{n \leq N} a(n) e^{\ii \theta_n} \Big |^2 \Big ].
  \]
\end{lem}

We need several lemmas controlling the distribution of the random model \eqref{eqn: Sl random}. The first one compares it explicitly to the Gaussian random variables $\mathcal Z_\ell$.  
\begin{lem}
  \label{lem: gaussian approx}
  Consider the increments $\mathcal Y_\ell=\mathcal S_\ell-S_{\ell-1}$, $1\leq \ell\leq\mathcal L$, defined in \eqref{eqn: Yl}, and the corresponding Gaussian variables $\mathcal N_\ell=\mathcal Z_\ell-\mathcal Z_{\ell-1}$ as in \eqref{eqn: Z_l}.
  There exists a constant $c>0$ such that, for any intervals $A$
  \[
  \PP\Big(\mathcal Y_\ell \in A\Big)
  =\PP\Big(\mathcal N_\ell\in A\Big)+\OO(e^{-c e^{t_{\ell-1}/2}}).
  \]
\end{lem}
\begin{proof}
  This follows similarly to \cite[Lemma 20]{argbourad20}, based on the Berry-Esseen estimate as stated in \cite[Lemma 19]{argbourad20}. The proof is actually more immediate because the covariances of $(\mathcal Y,\mathcal Y')$ and $(\mathcal N,\mathcal N')$ exactly coincide.
\end{proof}
The error $\OO(e^{-c e^{t_{\ell-1}/2}})$ is too large for $\ell=1$ for our purpose. In this case, we resort to the weaker estimate given in \cite[Lemma 18]{argbourad20}:
\begin{lem} \label{lem: gaussian approx 1}
  Consider the random variables $\mathcal S_\ell$ defined in \eqref{eqn: Sl random}.
  Let $\alpha\geq 1$. Then, for all $\Delta \geq \alpha $ and $|u| \leq 100\alpha \ell$, we have
  \[
  \mathbb{P}(\mathcal S_\ell \in [u, u + \Delta^{-1}]) \asymp \frac{1}{\Delta} \cdot \frac{e^{-u^2/{2v_\ell^2}}}{\sqrt{v_\ell}} .
  \]
\end{lem} 
When an exact Gaussian comparison is not needed, it is sometimes useful to bound the moments of the random model by Gaussian ones.
The following statement is a direct consequence of the fact that
\begin{equation}
  \label{eqn: MGF random}
  \E\left[\exp\Big(\lambda \sum_{x<p\leq y} \frac{\re \ a(p)X(p)}{p^{\sigma}}\Big)\right]=\prod_{x<p\leq y}I_0\Big(\frac{|a(p)|\lambda}{p^{\sigma}}\Big),
\end{equation}
where 
\begin{equation}
  \label{eqn: I_0}
  I_0(w)=\sum_{k=0}^\infty\frac{w^{2k}}{2^{2k}(k!)^2}.
\end{equation}
\begin{lem}
  \label{lem: bound moment gaussian}
  Let $\sigma\geq 1/2$ and $1\leq x<y$. For any complex numbers $a(p)$, we have for $k\in \N$,
  \begin{equation}
    \label{eqn: bound moment gaussian}
    \E\left[\Big(\sum_{x<p\leq y}\frac{\re \ a(p) e^{\ii \theta_p}}{p^{\sigma}} \Big)^{2k}\right]\leq
    \frac{(2k!)}{2^kk!}\mathfrak s(\sigma)^{2k},
  \end{equation}
  where 
  \begin{equation}
    \label{eqn: s2}
    \mathfrak s^2(\sigma)=\frac{1}{2}\sum_{x<p\leq y}\frac{|a(p)|^2}{p^{2\sigma}}.
  \end{equation}
  In particular, we have for any $\lambda\in \R$,
  \begin{equation}
    \label{eqn: bound mgf gaussian}
    \E\left[\exp\Big(\lambda \sum_{x<p\leq y} \frac{\re \ a(p)e^{i\theta_p}}{p^{\sigma}}\Big)\right]\leq \exp(\lambda^2\mathfrak s^2(\sigma)/2).
  \end{equation}
\end{lem}
To bound the moments of actual Dirichlet polynomials, we have the useful Lemma 3 of \cite{sou09}:
\begin{lem}
  \label{lem: sound moments}
  Let $2\leq x\leq T$. For any complex numbers $a(p)$ and $k\in \N$ with $x^k\leq T/\log T$, we have for $T$ large enough
  \[
  \E\left[\Big|\sum_{p\leq x}\frac{a(p)}{p^{\sigma+\ii \tau}}\Big|^{2k}\right]\ll k!\Big(\sum_{p\leq x}\frac{|a(p)|^2}{p^{2\sigma}}\Big)^k.
  \]
\end{lem}
Using Stirling, we get by picking $k=\lfloor V^2/v_\ell^2\rfloor$ (assuming that $V$ is not too large so the moment condition is satisfied) that
\[
\PP\Big(S_\ell-S_{\ell-1}>V\Big)\ll \sqrt{\frac{V^2}{v_\ell^2-v_{\ell-1}^2}} \exp\Big(-\frac{V^2}{2(v_\ell^2-v_{\ell-1}^2)}\Big).
\]

The asymptotics of the moments of the random model with all the prime powers present, as in the definition of $\mathcal Z_0$ in \eqref{eqn: Z0}, can be calculated explicitly.
\begin{lem}
  \label{lem: ak}
  Let $\delta>0$ and $\sigma=\frac{1}{2}+\frac{\delta}{\log T}$. Let $T_0$ such that $\log T_0=\oo(\log T)$. Then, as $T\to \infty$, we have for any $k>0$
  \begin{equation}\label{eqn: lemak}
  (\log T_0)^{-k^2}\E\left[\left|\prod_{p\leq T_0}\left(1-\frac{e^{\ii \theta_p}}{p^{\sigma}}\right)^{-1}\right|^{2k}\right]
  \sim 
  e^{\gamma k^2}\prod_{p}(1-p^{-1})^{k^2}\sum_{m\geq 0} \frac{|d_k(p^m)|^2}{p^{2m}},
  \end{equation}
  where $d_k(n)$ is the $k$-divisor function defined through the Taylor series
  \begin{equation}
    \label{eqn: divisor}
    (1-x)^{-k}=\sum_{n\geq 0}d_k(n)x^n=\sum_{n\geq 0} \frac{\Gamma(n+k)}{n!\Gamma (k)}x^n.
  \end{equation}
  (For an integer $k$, $d_k(n)$ is the number of ways of writing $n$ as the product of $k$ factors.) The Euler product on the right-hand side of \eqref{eqn: lemak} is $a_k$, cf.~\eqref{eqn: a_k}.
\end{lem}

\begin{proof}
  Let $2\leq X<T_0$. We will take the limit $T\to\infty$ then $X\to\infty$.
  We have by expanding the logarithm and the exponential
  \begin{align*}
    \E\left[\left|\prod_{X<p\leq T_0}\left(1-\frac{e^{\ii \theta_p}}{p^{\sigma}}\right)^{-1}\right|^{2k}\right]
    &=\prod_{X<p\leq T_0}\E\left[\exp\Big(-2k \re \log(1-e^{\ii \theta_p}p^{-\sigma})\Big)\right]\\
    &=\prod_{X<p\leq T_0}\E\left[\exp\Big(2k\Big( \frac{\cos \theta_p}{p^\sigma}+\frac{\cos 2\theta_p}{2p^{2\sigma}}+\OO(p^{-3\sigma})\Big)\Big)\right]\\
    &=\prod_{X<p\leq T_0}\left(1+\frac{k^2}{p^{2\sigma}} +\OO(p^{-3\sigma})\right)\\
    &=\exp\left(\sum_{X<p\leq T_0} \frac{k^2}{p^{2\sigma}} +\OO(X^{-1/2})\right).
  \end{align*}
  We can evaluate the sum over $p$ as in \eqref{eqn: variance} using the Prime Number Theorem. This gives as $T\to\infty$ and $X\to\infty$
  \[
  \E\left[\left|\prod_{X<p\leq T_0}\left(1-\frac{e^{\ii \theta_p}}{p^{\sigma}}\right)^{-1}\right|^{2k}\right]\sim
  (\log T_0)^{k^2}(\log X)^{-k^2},
  \]
  since $\delta\frac{\log T_0}{\log T}\to 0$ by assumption. It remains to evaluate $(\log X)^{-k^2}\E[|\prod_{p\leq X}(1-e^{\ii \theta_p}p^{-\sigma})^{-1}|^{2k}]$ in the limit $T\to\infty$ and $X\to\infty$.
  Taking $T\to\infty$ simplifies $\sigma\to 1/2$. Hence, writing the random Euler product as a random Dirichlet series yields
  \[
  \E\left[\left|\prod_{p\leq X}\left(1-\frac{e^{\ii \theta_p}}{p^\sigma}\right)^{-k}\right|^{2}\right]\sim
  \E\Big[\Big| \sum_{\substack{n\geq 1\\ p| n\Rightarrow p\leq X}}\frac{d_k(n)}{n^{1/2}}e^{\ii \theta_n}\Big|^{2}\Big]
  \]
  where $\theta_n$ is defined by $\theta_n=\theta_{p_1}^{a_1}\dots \theta_{p_r}^{a_r}$ if the prime factorization of $n$ is $p_1^{a_1}\dots p_r^{a_r}$.
  The right-hand side is easily evaluated by the orthogonality of the $e^{\ii\theta_n}$ and equals
  \[
  \sum_{\substack{n\geq 1\\ p| n\Rightarrow p\leq X}} \frac{|d_k(n)|^2}{n}=\prod_{p\leq X}\sum_{m\geq 0} \frac{|d_k(p^m)|^2}{p^{m}}.
  \]
  Mertens's third theorem, see for example \cite{Kou19}, states that
  $
  \log X \prod_{p\leq X}(1-p^{-1})\sim e^{-\gamma}
  $
  This implies that as $X\to\infty$
  \[
  (\log X)^{-k^2}\E\left[\left|\prod_{p\leq X}\left(1-\frac{e^{\ii \theta_p}}{p^{1/2}}\right)^{-1}\right|^{2k}\right]\sim e^{\gamma k^2}\prod_p (1-p^{-1})^{k^2}\sum_{m\geq 0} \frac{|d_k(p^m)|^2}{p^{m}},
  \]
  as claimed.
\end{proof}
The lemma implies some useful estimates. 
First, Markov's inequality yields
\begin{equation}
  \label{eqn: P Z_0}
  \PP(\mathcal Z_0>\alpha t_0)\leq  \E[e^{2\alpha \mathcal Z_0}](\log T_0)^{-2\alpha^2}
  \ll e^{\gamma \alpha^2}a_\alpha (\log T_0)^{-\alpha^2}.
\end{equation}
Second, the heuristic given \eqref{eqn: heuristic} remains almost unchanged on $\mathcal G_0=\{\mathcal Z_0-\alpha t_0\in [0,\sqrt{t_0}\}$.
Indeed, let's consider the change of probability defined by $\frac{\rd \widetilde{\PP}}{\rd \PP}=\frac{e^{2\alpha \mathcal Z_0}}{\E[e^{2\alpha \mathcal Z_0}]}$.
We have similar to \eqref{eqn: heuristic}:
\begin{align*}
  \PP(\{\mathcal Z_0+\mathcal Z_\mathcal L>\alpha t\}\cap\mathcal G_0 )
  &\asymp \frac{e^{-\alpha^2 t}}{\alpha \sqrt{t}} e^{-\alpha^2 \mathfrak s \log_\mathcal L t-2\delta \alpha^2} \cdot e^{-\alpha^2t_0} \E[e^{2\alpha \mathcal Z_0}]\cdot \widetilde{\PP}(\mathcal Z_0-\alpha t_0\in [0,\sqrt{t_0}])\\
  &\asymp \frac{e^{-\alpha^2 t}}{\alpha \sqrt{t}} a_\alpha e^{-\alpha^2 \mathfrak s \log_\mathcal L t-2\delta \alpha^2 +\gamma\alpha^2} \cdot \widetilde{\PP}(\mathcal Z_0-\alpha t_0\in [0,\sqrt{t_0}]),
\end{align*}
by Lemma \ref{lem: ak}.
Under $\widetilde{\PP}$, the random variables $e^{\ii\theta_p}$ remain independent and $\widetilde{\E}[\mathcal Z_0]=\alpha t_0+\oo(1)$. 
Therefore, by the Berry-Esseen theorem, we have that $\widetilde{\PP}(\mathcal Z_0-\alpha t_0\in [0,\sqrt{t_0}])\asymp 1$. 
We conclude that
\begin{equation}
  \label{eqn: remark a}
  \PP(\{\mathcal Z_0+\mathcal Z_\mathcal L>\alpha t\}\cap\mathcal G_0 )\asymp\PP(\mathcal Z_0+\mathcal Z_\mathcal L>\alpha t).
\end{equation}

\bibliographystyle{alpha}
\bibliography{selbergLB.bib}
  
\end{document}